\input ./style/arxiv-general.cfg
\documentclass[aop,MSNbibl,seceqn,dvips]{arximspdf}
\makeatletter
   \@ifpackageloaded{graphicx}{}{\usepackage{graphicx}}
\makeatother
\doi{10.1214/14-AOP954} % kopijuoti is laisko
\volume{43}
\issue{6}
\pubyear{2015}
\firstpage{3006}
\lastpage{3051}
\docsubty{FLA}

\makeatletter
\DeclareFontFamily{U}{mathx}{\hyphenchar\font45}
\DeclareFontShape{U}{mathx}{m}{n}{
      <5> <6> <7> <8> <9> <10>
      <10.95> <12> <14.4> <17.28> <20.74> <24.88>
      mathx10
      }{}
\DeclareSymbolFont{mathx}{U}{mathx}{m}{n}
\DeclareFontSubstitution{U}{mathx}{m}{n}
\DeclareMathAccent{\widecheck}{0}{mathx}{"71}
\DeclareMathAccent{\wideparen}{0}{mathx}{"75}

\newcommand{\iint}{\int\!\!\!\int}
\newcommand{\rrvert}{\vert}
\newcommand{\llvert}{\vert}
\newcommand{\eqref}[1]{(\ref{#1})}
\newtheorem{theorem}{Theorem}[section]
\newtheorem{corollary}[theorem]{Corollary}
\newtheorem{lemma}[theorem]{Lemma}
\newtheorem{proposition}[theorem]{Proposition}

\newtheorem{lemmas}{Lemma}[section]
\newtheorem{propositionn}[lemmas]{Proposition}

\newproclaim{definition}[theorem]{Definition}

\newproclaim{assumption}[theorem]{Assumption}

\newproclaim{remark}[theorem]{Remark}

\newproclaim{example}[theorem]{Example}

\newcommand{\E}{\mathbb{E}}
\newcommand{\W}{\dot{W}}
\newcommand{\ud}{\mathrm{d}}
\newcommand{\Ito}{It\^{o} }
\newcommand{\Itos}{It\^{o}'s }

\newcommand{\calB}{\mathcal{B}}
\newcommand{\calF}{\mathcal{F}}
\newcommand{\calG}{\mathcal{G}}
\newcommand{\calK}{\mathcal{K}}
\newcommand{\calH}{\mathcal{H}}
\newcommand{\calI}{\mathcal{I}}
\newcommand{\calL}{\mathcal{L}}
\newcommand{\calM}{\mathcal{M}}
\newcommand{\calN}{\mathcal{N}}
\newcommand{\calP}{\mathcal{P}}
\newcommand{\bbN}{\mathbb{N}}

\newcommand{\R}{\mathbb{R}}
\newcommand{\Erf}{\ensuremath{\mathrm{erf}}}
\newcommand{\Erfc}{\ensuremath{\mathrm{erfc}}}

\newcommand{\LIP}{\operatorname{Lip}}

\newcommand{\Vip}{\overline{\varsigma}}
\newcommand{\vv}{\varsigma}
\newcommand{\sd}{\beta}
\makeatother

\begin{document}
\begin{frontmatter}

%\dochead{}
\title{Moments and growth indices
for the nonlinear stochastic heat equation
with rough initial~conditions\thanksref{T1}}
\runtitle{Moments in the stochastic heat equation}
\thankstext{T1}{Supported in part by the Swiss National Foundation
for Scientific Research.}

\begin{aug}
\author[A]{\fnms{Le}~\snm{Chen}\corref{}\ead[label=e1]{chenle02@gmail.com}}
\and
\author[B]{\fnms{Robert C.}~\snm{Dalang}\ead[label=e2]{robert.dalang@epfl.ch}}
\runauthor{L. Chen and R.~C. Dalang}
\affiliation{\'Ecole Polytechnique F\'ed\'erale de Lausanne}
%\dedicated{}
\address[A]{Department of Mathematics\\
University of Kansas\\
405 Snow Hall\\
1460 Jayhawk Blvd\\
Lawrence, Kansas 66045-7594\\
USA\\
\printead{e1}} %adresu isvedimo komanda gale!
\address[B]{Institut de math\'ematiques\\
\'Ecole Polytechnique F\'ed\'erale de Lausanne\\
Station 8\\
CH-1015 Lausanne\\
Switzerland\\
\printead{e2}}
\end{aug}

% HISTORY:
\received{\smonth{6} \syear{2013}}
\revised{\smonth{2} \syear{2014}}
%\accepted{\smonth{} \syear{}}

% ABSTRACT
%
\begin{abstract}
We study the nonlinear stochastic heat equation in
the spatial domain $\R$, driven by space--time white noise. A central special
case is the parabolic Anderson model. The initial condition is taken to
be a
measure on $\R$, such as the Dirac delta function, but this measure
may also
have noncompact support and even be nontempered (e.g., with
exponentially growing tails). Existence and uniqueness of a random
field solution is proved without
appealing to Gronwall's lemma, by keeping tight control over moments in the
Picard iteration
scheme. Upper bounds on all $p$th moments $(p\ge2)$ are obtained as
well as
a lower bound on second moments.
These bounds become equalities for the parabolic Anderson model when
$p=2$. We
determine the growth indices introduced by Conus and Khoshnevisan
[\textit{Probab. Theory Related Fields} \textbf{152} (2012) 681--701].
\end{abstract}

% KEYWORDS
%Pirmas kwd is didziosios raides
%
\end{frontmatter}

%s1 #&#
\section{Introduction}
The stochastic heat equation
%
%e1.1 #&#
\begin{eqnarray}
\label{E2:Heat} \cases{ %
\displaystyle \biggl(\frac{\partial}{\partial t} - \frac{\nu}{2}
\frac{\partial^2 }{\partial x^2} \biggr) u(t,x) = \rho\bigl(u(t,x)\bigr)
 \dot{W}(t,x),&\quad $x\in\R,
t \in\R_+^*$, \vspace*{2pt}
\cr
u(0,\cdot) = \mu(\cdot),}
\end{eqnarray}
where $\dot{W}$ is space--time white noise, $\rho(u)$ is globally Lipschitz,
$\mu$ is the initial data, and $\R_+^*=\, ]0,\infty[$, has been intensively
studied during the last three decades by many authors: See
\cite
{AmirCorwinQuastel11Stationary,BertiniCancrini94Intermittence,BorodinCorwin11MP,CarmonaMolchanov94,ConusEct12Initial,ConusKhosh10Weak,ConusKhosh10Farthest,DalangMueller09Intermittency,FoondunKhoshnevisan08Intermittence}
for the intermittency problem,
\cite{DalangKhNualart07HittingAdditive,DalangKhNualart09HittingMult}
for probabilistic potential theory,
\cite{SanSoleSarra99Path,SanSoleSarra99Holder}
for regularity of the solution and
\cite
{DalangEtc08Minicourse,Mueller91Support,MytnikPerkins11Pathwise,PospisilTribe07Parameter,Shiga94Two}
for several other properties.
The important special
case $\rho(u)=\lambda u$ is called \textit{the parabolic Anderson model}
\cite{CarmonaMolchanov94}. Our
work focuses on \eqref{E2:Heat} with general deterministic initial data
$\mu$,
and we study how the initial data affects the moments and asymptotic
properties of the solution.

% % % Existence

For the existence of random field solutions (see Definition~\ref{D2:Solution}
below) to \eqref{E2:Heat}, the case where
the initial data
$\mu$ is a bounded and measurable function is covered by the classical
theory of
Walsh \cite{Walsh86}.
Initial data that is more irregular than this also appears the
literature. For
instance, when $\mu$ is a positive Borel measure on $\R$ such that
%
%e1.2 #&#
\begin{equation}
\label{E1:BC-InitD} \sup_{t\in[0,T]} \sup_{x\in\R}
\sqrt{t} \bigl(\mu*G_\nu(t,\circ) \bigr) (x)<\infty \qquad\mbox{for all
$T>0$},
\end{equation}
where $*$ denotes convolution in the spatial variable and
%
%e1.3 #&#
\begin{equation}
% \label{E2:1G-Heat}
G_\nu(t,x) := \frac{1}{\sqrt{2\pi\nu t}} \exp \biggl\{-
\frac
{x^2}{2\nu
t} \biggr\}, \qquad(t,x)\in\R_+^*\times\R.
\end{equation}
Bertini and Cancrini
\cite{BertiniCancrini94Intermittence} gave an ad-hoc definition of
solution for the parabolic
Anderson model via a smoothing of the space--time white
noise and\break a Feynman--Kac type formula. Their analysis depended heavily on
properties of the local times of Brownian bridges. Recently, Conus and
Khoshnevisan \cite{ConusKhosh10Weak} have constructed a weak solution
defined through
certain norms on random fields. In particular, their solution is
defined for almost all $(t,x)$, but
not at specific
$(t,x)$. Their initial data has to verify certain
technical conditions, which are satisfied by the Dirac delta function
in some of
their cases. More recently, Conus, Joseph, Khoshnevisan and Shiu
\cite{ConusEct12Initial} also studied random field
solutions. In particular, they require the initial data to be a finite measure
of compact support.

After the basic questions of existence, the asymptotic properties of
the solution are of particular interest, in part because the solution
exhibits intermittency properties. More precisely, define the \textit
{upper and
lower Lyapunov exponents} %for constant initial data (the Lebesgue
%measure)
as follows:
%
%e1.4 #&#
\begin{eqnarray}
% \label{E1:Lyapunov}
\overline{m}_p(x)&:=&\mathop{\lim\sup}_{t\rightarrow+\infty}
\frac{\log
\E [|u(t,x)|^p ]}{t},
\nonumber
\\[-8pt]
\\[-8pt]
\nonumber
\underline{m}_p(x)&:=&\mathop{\lim
\inf}_{t\rightarrow+\infty} \frac{\log
\E [|u(t,x)|^p ]}{t}.
\end{eqnarray}
When the initial data is constant, these two exponents do not depend on $x$.
In this case, following Bertini and Cancrini
\cite{BertiniCancrini94Intermittence}, we
say that the solution is \textit{intermittent} if
$m_n:=\underline{m}_n =\overline{m}_n$ for all $n\in\bbN$ and the following
strict inequalities are satisfied:
%
%e1.5 #&#
\begin{equation}
\label{E1:Intermit} m_1 < \frac{m_2}{2}<\cdots<\frac{m_n}{n}<
\cdots.
\end{equation}
Carmona and Molchanov gave the following definition
\cite{CarmonaMolchanov94}, Definition III.1.1, on page~55.

%de1.1 #&#
\begin{definition}\label{D2:Intermit}
Let $p$ be the smallest integer for which $m_p>0$. If $p<\infty$, then we
say that the solution $u(t,x)$ exhibits \emph{\textup{(}asymptotic\textup{)}
intermittency of
order~$p$}, and if $p=2$, then it exhibits \emph{full intermittency}.
\end{definition}

Carmona and Molchanov \cite{CarmonaMolchanov94} showed that full
intermittency implies the
intermittency defined by
\eqref{E1:Intermit} (see \cite{CarmonaMolchanov94}, Theroem~III.1.2, on page~55).
This mathematical definition of intermittency is related to the
property that
the solutions are close to zero in vast regions of space--time but develop
high peaks on some small ``islands.''
For the parabolic Anderson model, this property has been well
studied;
see \cite{CarmonaMolchanov94,CrMountfordSh02Lyapunov}
for a discrete formulation and
\cite
{BertiniCancrini94Intermittence,FoondunKhoshnevisan08Intermittence,DalangMueller09Intermittency}
for the continuous formulation.
Further general discussion of the intermittency property can be found
in \cite{Zeldovich90Almighty}.

When the initial data are not homogeneous, in particular, when they
have certain
decrease at infinity, Conus and Khoshnevisan
\cite{ConusKhosh10Farthest} defined the following
\textit{lower and upper exponential growth indices}:
%
%e1.6 #&#
%e1.7 #&#
\begin{eqnarray}
\label{E1:GrowInd-0} \underline{\lambda}(p)&:= & \sup \biggl\{\alpha>0
\dvtx\mathop{\lim\sup}_{t\rightarrow\infty} \frac{1}{t}\sup_{|x|\ge\alpha t} \log\E
\bigl(\bigl|u(t,x)\bigr|^p \bigr) >0 \biggr\} ,
\\
\label{E1:GrowInd-1} \overline{\lambda}(p)&:= & \inf \biggl\{\alpha>0
\dvtx\mathop{\lim\sup}_{t\rightarrow\infty} \frac{1}{t}\sup_{|x|\ge\alpha t} \log\E
\bigl(\bigl|u(t,x)\bigr|^p \bigr) <0 \biggr\}.
\end{eqnarray}
These quantities are of interest because they give information about the
possible locations of high peaks, and how they propagate away from the origin.
Indeed, if $\underline{\lambda}(p)=\overline{\lambda}(p)=: \lambda(p)$,
then there will be high peaks at time $t$ inside $[-\lambda(p) t,
\lambda(p)
t]$, but no
peaks outside of this interval.
Conus and Khoshnevisan~\cite{ConusKhosh10Farthest} proved in particular
that if
the initial
data $\mu$ is a nonnegative, lower
semicontinuous function with compact support of positive Lebesgue
measure, then for the
Anderson model, %($\rho(u)=\lambda u$),
%
%e1.8 #&#
\begin{equation}
\label{e1.7a} \frac{\lambda^2}{2\pi} \le \underline{\lambda}(2) \le\overline{
\lambda}(2) \le\frac{\lambda^2}{2}.
\end{equation}

In this paper, we improve the existence result by working under a
much weaker condition on the initial data, namely, $\mu$ can be any signed
Borel measure over $\R$ such that
%
%e1.9 #&#
\begin{equation}
\label{E1:J0finite} \int_\R e^{-a x^2} |\mu|(\ud x) <+
\infty \qquad\mbox{for all $a>0$} ,
\end{equation}
where, from the Jordan decomposition, $\mu=\mu_+-\mu_-$ where
$\mu_\pm$ are two nonnegative Borel measures
with disjoint support and $|\mu|:= \mu_++\mu_-$.
Note that the condition \eqref{E1:J0finite} is equivalent to
\[
\bigl(|\mu| * G_\nu(t,\cdot) \bigr) (x) <+\infty\qquad \mbox{for all
$t>0$ and $x\in\R$},
\]
which means that under condition \eqref{E1:J0finite}, the solution to the
homogeneous heat equation with initial data $\mu$ is well defined for
all time.

On the one hand, condition \eqref{E1:J0finite} allows for measure-valued
initial data, such as the Dirac delta function, and
Proposition~\ref{P2:D-Delta} below shows that
initial data cannot be extended beyond measures to other Schwartz distributions,
even with compact support.
On the other hand, the condition \eqref{E1:J0finite} permits certain
exponential growth at infinity.
For instance, if $\mu(\ud x)=f(x)\,\ud x$, then $f(x)=\exp
(a|x|^p )$,
$a>0$, $p\in\,]0,2[$ (i.e., exponential growth at $\pm\infty$), will satisfy
this condition. Note that the case where the initial data is a continuous
function with linear exponential growth (i.e., $p=1$) has been
considered by
many authors; see \cite
{MytnikPerkins11Pathwise,PospisilTribe07Parameter,Shiga94Two} and the
references therein.

Next, we obtain estimates for the moments $\E(|u(t,x)|^p)$ with both
$t$ and
$x$ fixed for all even integers $p\ge2$ (see Theorem~\ref{T2:ExUni}). In
particular, for the
parabolic Anderson model, we give an explicit formula for the second
moment of
the solution.
When the initial data is either Lebesgue measure or the Dirac delta
function, we
give explicit formulas for the two-point correlation functions [see
\eqref{E2:TP-Lebesgue} and~\eqref{E2:TP-Delta} below], which
can be compared to the integral form given by Bertini and
Cancrini~\cite{BertiniCancrini94Intermittence}, Corollaries 2.4 and
2.5 (see also Remark~\ref{R2:TP-Lebesgue}
below).

Recently, Borodin and Corwin \cite{BorodinCorwin11MP} also obtained the
moment formulas for the parabolic Anderson model in the case where the initial
data is the Dirac delta function. When $p=2$, we obtain the same explicit
formula. For $p>2$, their $p$th moments are represented by multiple contour
integrals.
Our methods are very different from theirs: They
approximate the continuous system by a discrete one. Our formulas allow more
general initial data than the Dirac delta function, and are useful for
establishing other properties, concerning for instance growth indices
and sample path regularity.

Our proof of existence is based on the standard Picard iteration
scheme. The
main difference from the conventional situation is that instead of applying
Gronwall's lemma to bound the second moment from above, we keep tight
control over the
sequence of second moments in the Picard iteration scheme. In the case of
the parabolic Anderson model, this directly gives an explicit formula,
and for
more general functions $\rho$ it gives good bounds.
Note that series representations of the moments are obtained in
\cite{DalangMT06FKT}, yielding a Feynman--Kac-type
formula.

Concerning growth indices, we improve \eqref{e1.7a}
by giving upper bounds on $\overline{\lambda}(p)$ for general functions
$\rho$,
and, in the parabolic Anderson model, by
showing that $\underline{\lambda}(2)= \overline{\lambda}(2) =
\lambda^2/2$ when $\mu$ is a nonnegative measure with compact support (see
Theorem~\ref{T2:Growth}), and we
extend this result to a more general class of measure-valued initial data
(not necessarily with compact support). This is possible mainly thanks
to our
explicit formula for the second moment.
Our result implies in particular that with regard to the propagation of high
peaks, an initial condition with tails that decrease at a sufficiently high
exponential rate [as least as fast as $e^{-\beta|x|}$ with
$\beta\ge\lambda^2/(2\nu)$] produces the same behavior as a
compactly supported
one.

% \paragraph{Outline of the chapter}
This paper is organized as follows:
All the main results of this paper are stated in Section~\ref{S:MR}.
In particular, in Section~\ref{SS:Main-ExUnMm}, we define the
notion of random field solution of \eqref{E2:Heat}, and then show, assuming
existence of the solution, that one obtains readily formulas for the second
moments in the case of the Anderson model. Then we state and prove our theorem
on existence, uniqueness and moment estimates, discuss various particular
initial conditions, including Lebesgue measure and the Dirac delta
function, and
we show that existence is not possible if the initial condition is
rougher than
a measure.
In Section~\ref{SS:Main-Exp}, we state the results about the
growth
indices.
Proofs of the results in Sections \ref{SS:Main-ExUnMm} and~\ref{SS:Main-Exp}
are given in Sections \ref{S:Proof-ExUnMm} and~\ref{S:Proof-Exp},
respectively.
Finally, in Section~\ref{sec4}, we gather various calculations that
are used
throughout the paper.

%s2 #&#
\section{Main results}\label{S:MR}

Let $\calM (\R )$ be the set of locally finite (signed)
Borel measures
over $\R$.
Let $\calM_H (\R )$ be the set of signed Borel measures over
$\R$
satisfying~\eqref{E1:J0finite}.
Denote the solution to the homogeneous equation
%
%e2.1 #&#
\begin{equation}
\label{E2:Heat-home} \cases{ %
\displaystyle \biggl(\frac{\partial}{\partial t} - \frac{\nu}{2}
\frac{\partial^2 }{\partial x^2} \biggr) u(t,x) = 0,&\quad
$x\in\R, t \in\R_+^*$, \vspace*{2pt}
\cr
u(0,\cdot) = \mu(\cdot) ,}
\end{equation}
by
\[
J_0(t,x) := \bigl(\mu* G_{\nu}(t,\cdot) \bigr) (x) =
\int_\R G_\nu (t,x-y )\mu(\ud y).
\]
%
% Note that $J_0(t,x)$ is well-defined if $\mu\in\calM_H(\R)$.

%s2.1 #&#
\subsection{Existence, uniqueness and
moments}\label{SS:Main-ExUnMm}

Let $W= \{
W_t(A), A\in\calB_b (\R ),t\ge0  \}$
be a space--time white noise
defined on a complete probability space $(\Omega,\calF,P)$, where
$\calB_b (\R )$ is the
collection of Borel measurable sets with finite Lebesgue measure.
Let
\[
\calF_t = \sigma \bigl(W_s(A), 0\le s\le t, A\in
\calB_b (\R ) \bigr)\vee \calN,\qquad t\ge0,
\]
be the natural filtration of $W$ augmented by the $\sigma$-field
$\calN$
generated by all $P$-null sets in $\calF$.
In the following, we fix the filtered
probability space $ \{\Omega,\calF,\{\calF_t, t\ge0\},P
\}$.
We use $\Vert\cdot\Vert_p$ to denote the
$L^p(\Omega)$-norm ($p\ge1$).
With this setup, $W$ becomes a worthy martingale measure in the sense
of Walsh
\cite{Walsh86}, and $\iint_{[0,t]\times\R}X(s,y) W(\ud s,\ud y)$ is
well defined in this reference for a suitable class of random fields
$ \{X(s,y), (s,y)\in\R_+\times\R \}$.

We can formally rewrite the spde \eqref{E2:Heat} in the integral form:
%
%e2.2 #&#
\begin{equation}
\label{E2:WalshSI} u(t,x) = J_0(t,x) + I(t,x),
\end{equation}
where
\[
% \label{E2:WalshSI-I}
I(t,x):=\iint_{[0,t]\times\R} G_\nu (t-s,x-y )
\rho \bigl( u (s,y ) \bigr) W (\ud s,\ud y ).
\]
We use the convention that $G_\nu(t,\cdot)\equiv0$ if $t\le0$. Hence,
$[0,t]\times\R$ in the stochastic integral above can be
replaced by $\R_+\times\R$. In the following, we will
use $\star$ to denote the simultaneous convolution in both space
and time variables,

%de2.1 #&#
\begin{definition}\label{D2:Solution}
A process $u= (u(t,x), (t,x)\in\R_+^*\times\R )$ is
called a
\textit{random field solution} to
\eqref{E2:WalshSI} if:
\begin{longlist}[(1)]
\item[(1)]$u$ is adapted, that is, for all $(t,x)\in\R_+^*\times\R$,
$u(t,x)$ is
$\calF_t$-measurable;
\item[(2)]$u$ is jointly measurable with respect to
$\calB (\R_+^*\times\R )\times\calF$;
\item[(3)]$ (G_\nu^2 \star\Vert\rho(u)\Vert_2^2
)(t,x)<+\infty$
for all $(t,x)\in\R_+^*\times\R$, and
the function
$(t,x)\mapsto I(t,x)$ mapping $\R_+^*\times\R$ into
$L^2(\Omega)$ is continuous;
\item[(4)]$u$ satisfies \eqref{E2:WalshSI} a.s.,
for all $(t,x)\in\R_+^*\times\R$.
\end{longlist}
\end{definition}

Notice that the random field is only defined for $t>0$, which is
natural since
at time $t=0$, the solution is defined to be a measure.

According to property (3) in this definition, proving the existence of a
random field solution
requires some estimates on its moments. On the other hand, if we \emph
{assume}
existence, then one can readily obtain moment formulas or bounds.
Indeed, consider for example, the parabolic Anderson
model, and set
\[
f(t,x)=\E\bigl(u(t,x)^2\bigr).
\]
For $(t,x)\in\R_+^*\times\R$ and $n\in\bbN$, we define
%
%e2.3 #&#
\begin{eqnarray}
\label{E:Def-Ln} %
\calL_0(t,x)&=&
\calL_0 (t,x;\nu,\lambda ):= \lambda^2
G_\nu^2 (t,x ) = \frac{\lambda^2}{\sqrt{4\pi\nu t}} G_{{\nu}/{2}}(t,x),
\nonumber
\\[-8pt]
\\[-8pt]
\nonumber
\calL_n(t,x)&=&\calL_n (t,x;\nu,\lambda ) := (\underbrace
{\calL _0 \star \cdots\star\calL_0}_{n+1\ \mathrm{times\ of}\ \calL_0})
(t,x)\qquad \mbox{for $n\ge1$.}
\end{eqnarray}
Then by \eqref{E2:WalshSI} and \Itos
isometry, $f(t,x)$
satisfies the integral equation
%
%e2.4 #&#
\begin{equation}
\label{E2:fint} f(t,x)= J_0^2(t,x) + (f\star
\calL_0 ) (t,x).
\end{equation}
Apply this relation recursively:
\begin{eqnarray*}
f(t,x) &=& J_0^2(t,x) + \bigl( \bigl[J_0^2
+ (f\star\calL_0 ) \bigr]\star \calL_0 \bigr) (t,x)
\\
&=& J_0^2(t,x) + \bigl( J_0^2
\star\calL_0 \bigr) (t,x) + (f\star \calL_1 ) (t,x)
\\
&\vdots&
\\
&=& J_0^2(t,x) + \sum_{i=0}^{n-1}
\bigl(J_0^2\star\calL_i \bigr) (t,x) + (f
\star\calL_n ) (t,x).
\end{eqnarray*}
It follows from \eqref{E:KS-LB} below and Definition~\ref
{D2:Solution}(3) that
$ (f\star\calL_n )(t,x)$ converges to~$0$ as $n\rightarrow
\infty$,
and the sum converges to $ (J_0^2\star
\calK )(t,x)$, where
%
%e2.5 #&#
\begin{equation}
\label{E:Def-K} \calK(t,x)=\calK(t,x;\nu,\lambda):= \sum
_{i=0}^\infty\calL _i(t,x;\nu ,\lambda).
\end{equation}
Thus,
%
%e2.6 #&#
\begin{equation}
\label{E:SecMom} \E \bigl(u(t,x)^2 \bigr) =J_0^2(t,x)+
\bigl(J_0^2\star\calK \bigr) (t,x).
\end{equation}

A central observation is that $\calK(t,x)$ can be computed explicitly,
as we
now show.
Let
\begin{eqnarray*}
\Phi(x) &=& \int_{-\infty}^x (2\pi)^{-1/2}e^{-y^2/2}
\,\ud y,\qquad \Erf(x)=\frac{2}{\sqrt{\pi}}\int_0^x
e^{-y^2}\,\ud y,\\
 \Erfc(x)&=&1-\Erf(x).
\end{eqnarray*}
Clearly,
\begin{eqnarray*}
\Phi(x) &=& \tfrac{1}{2} \bigl(1+\Erf(x/\sqrt{2}) \bigr),\qquad \Erf(x) = 2\Phi(
\sqrt{2} x) -1 ,\\
 \Erfc(x)&=&2 \bigl(1-\Phi(\sqrt{2} x) \bigr).
\end{eqnarray*}
Let $\Gamma(\cdot)$ be Euler's gamma function \cite{NIST2010}.

%pr2.2 #&#
\begin{proposition}\label{P2:K}
Let\vspace*{1pt} $b=\frac{\lambda^2}{\sqrt{4\pi\nu}}$.
For all $n\in\bbN$ and $(t,x)\in\R_+^*\times\R$,
let $\calL_n(t,x)$ and $\calK(t,x)$ be defined in \eqref{E:Def-Ln} and
\eqref{E:Def-K}, respectively. Then
%
%e2.7 #&#
\begin{equation}
\label{E:KS-LB} \calL_n(t,x) = G_{{\nu}/{2}}(t,x)
\frac{ (b\sqrt{\pi} )^{n+1}}{\Gamma (
{(n+1)}/{2} )} t^{{(n-1)}/{2}}= \calL_0(t,x)
B_n(t),
\end{equation}
with $B_n(t):=\pi^{{(n+1)}/{2}} b^n
t^{{n}/{2}}/\Gamma (\frac{n+1}{2} )$, and
%
%e2.8 #&#
\begin{equation}
\label{E:K} \calK(t,x)=G_{{\nu}/{2}}(t,x) \biggl(\frac{\lambda^2}{\sqrt{4\pi\nu t}}+
\frac{\lambda^4}{2\nu} e^{{\lambda^4 t}/{(4\nu)}}\Phi \biggl(\lambda^2 \sqrt{
\frac{t}{2\nu}} \biggr) \biggr).
\end{equation}
Furthermore,
%
%e2.9 #&#
\begin{equation}
\label{E:KS-Kn-K0} (\calK\star\calL_0 ) (t,x) = \calK(t,x) -
\calL_0(t,x) ,
\end{equation}
and $\sum_{n=0}^\infty (B_n(t) )^{1/m}<+\infty$, for all
$m\in\bbN^*$.
\end{proposition}
\begin{pf}
Since $\Gamma(1/2)=\sqrt{\pi}$ (see \cite{NIST2010}, Equation~5.4.6, page~137),
the equation \eqref{E:KS-LB} clearly
holds for $n=0$.
Suppose by induction that it is true for $n$.
Using the semigroup property of the heat kernel,
\begin{eqnarray*}
\calL_{n+1}(t,x)& =& (\calL_n \star\calL_0 )
(t,x)\\
&=&G_{{\nu}/{2}}(t,x) b \frac{ (b\sqrt{\pi} )^{n+1}}{\Gamma (
{(n+1)}/{2} )} \int_0^t
s^{-1/2}(t-s)^{{(n-1)}/{2}}\,\ud s.
\end{eqnarray*}
Therefore, \eqref{E:KS-LB} is obtained by using the Beta integral
(see \cite{NIST2010}, (5.12.1), page~142)
%
%e2.10 #&#
\begin{equation}\qquad
\label{E2:BetaInt} \int_0^t s^{-1/2}(t-s)^{{(n-1)}/{2}}
\,\ud s= t^{n/2} \frac{\Gamma(1/2)\Gamma ({(n+1)}/{2} )}{\Gamma
({(n+2)}/{2}
 )}\qquad \mbox{for $t>0$.}
\end{equation}

Because
\[
e^{x^2} \Erf(x) = \sum_{n=1}^\infty
\frac{
x^{2n-1}}{\Gamma ({(2n+1)}/{2} )} \quad \mbox{and}\quad e^{x^2} = \sum
_{n=1}^\infty \frac{x^{2(n-1)}}{\Gamma ({2n}/{2} )}
\]
(see \cite{NIST2010}, Equation~7.6.2, on page~162, for the first equality),
we see that for $x>0$,
\[
e^{x^2} \bigl(1+\Erf(x) \bigr)= \sum_{n=1}^{\infty}
\frac{x^{n-1}}{\Gamma (
{(n+1)}/{2} )} =-\frac{1}{\sqrt{\pi} x}+\sum_{n=0}^{\infty}
\frac{x^{n-1}}{\Gamma ({(n+1)}/{2} )}.
\]
Move the term $-1/(\sqrt{\pi} x)$ to the left-hand side,
choose $x = \sqrt{\pi b^2 t}$, and then multiply by $\pi b^2 G_{\nu
/2}(t,x)$ on
both sides. Hence, from \eqref{E:KS-LB}, we see that\vspace*{-1pt}
\begin{eqnarray*}
G_{{\nu}/{2}}(t,x) \biggl[\frac{b}{\sqrt{t}} + 2\pi b^2
e^{\pi
b^2 t} \Phi \bigl(\sqrt{2\pi b^2 t} \bigr) \biggr] &=&
G_{{\nu}/{2}}(t,x) \sum_{n=0}^{\infty}
\frac{(b\sqrt{\pi})^{n+1}}{\Gamma ({(n+1)}/{2} )} t^{{(n-1)}/{2}} \\[-1pt]
&= &\sum_{n=0}^\infty
\calL_n(t)=\calK(t,x),\vspace*{-2pt}
\end{eqnarray*}
which proves \eqref{E:K}.

Formula \eqref{E:KS-Kn-K0} is a direct consequence of \eqref{E:Def-K}.
Finally, fix $m\in\bbN^*$.
Apply the ratio test:\vspace*{-1pt}
%
%%e2.11 #&#
\begin{eqnarray}
\label{E:RT} \frac{ (B_n(t) )^{1/m}}{ (B_{n-1}(t) )^{1/m}} &=&
(\sqrt{\pi t} b )
^{{1}/{m}} \biggl(
\frac{\Gamma ({n}/{2} )}{\Gamma (
{(n+1)}/{2} )} \biggr)^{{1}/{m}}
\nonumber
\\[-9pt]
\\[-9pt]
\nonumber
&\approx &(\sqrt{\pi t} b )^{{1}/{m}}
\biggl(\frac{2}{n} \biggr)^{{1}/{(2m)}}\rightarrow 0\qquad \mbox{as }n
\rightarrow\infty,\vspace*{-1pt}
\end{eqnarray}
where we have used \cite{NIST2010}, Equation~5.11.12, page~141, for the ratio of
the two gamma functions.
Therefore, $\sum_{n=0}^\infty (B_n(t) )^{1/m}<+\infty$. This
completes the proof.
\end{pf}

%re2.3 #&#
\begin{remark}[(Moment formula via the Fourier and Laplace transforms)]
\label{R:Transform}
If we assume the existence of a random field solution, then under additional
assumptions, one can also obtain the moment formula by using Fourier
and Laplace
transforms. In particular, consider the case where $\rho(u) = \lambda
u$. Then
$f(t,x)=\E[u(t,x)^2]$ satisfies equation \eqref{E2:fint}.
Assume that the double transform---the Fourier transform in $x$ and Laplace
transform in $t$---of $J_0^2(t,x)$ exists.
Note that this assumption is rather strong: If the initial
data has exponential growth, for example, $\mu(\ud x) = e^{\sd|x|}\,\ud
x$ with
$\sd>0$, then $J_0(t,x)$ has two exponentially growing tails [see
\eqref{E2:Eac2}], and hence the Fourier transform of
$J_0^2(t,x)$ in $x$ does not exist in the sense of tempered distributions.
Apply the Fourier transform in $x$ and then the Laplace transform in
$t$ on both
sides of \eqref{E2:fint}:\vspace*{-1pt}
\[
\calL\calF [f ](z,\xi) = \calL\calF \bigl[J_0^2
\bigr](z,\xi) + \lambda^2 \calL\calF \bigl[G_\nu^2
\bigr](z,\xi) \calL\calF [f ](z,\xi).\vspace*{-1pt}
\]
Solving for $\calL\calF[f](z,\xi)$, we see that\vspace*{-1pt}
\[
\calL\calF [f ](z,\xi) = \calL\calF \bigl[J_0^2
\bigr](z,\xi) + \frac{\lambda^2\calL\calF [G_\nu^2 ](z,\xi)}{1-\lambda
^2\calL
\calF
[G_\nu^2 ](z,\xi)} \calL\calF \bigl[J_0^2
\bigr](z,\xi).\vspace*{-1pt}
\]
Apply the Fourier and Laplace transforms to $G_\nu^2(t,x)$ as follows
(see \cite{Erdelyi1954-I}, page~135):\vspace*{-1pt}
\begin{eqnarray*}
\calF \bigl[G_\nu^2(t,\cdot) \bigr](\xi) &=&
\frac{\exp (-\nu t
|\xi|^2/4 )}{\sqrt{4\pi\nu t}}\quad \mbox{and}\\
 \calL\calF\bigl[G_\nu^2
\bigr](z,\xi) &=& \frac{1}{\sqrt{4\nu z + |\xi|^2\nu^2}}, \qquad\Re[z]>0.
\end{eqnarray*}
Now apply the inverse Laplace transform (see \cite{Erdelyi1954-I}, (4)
on page~233)
to see that
\begin{eqnarray*}
&&\calL^{-1} \biggl[\frac{\lambda^2\calL\calF [G_\nu^2
](z,\xi)}{
1-\lambda^2\calL\calF
[G_\nu^2 ](z,\xi)} \biggr]\\
&&\qquad= \calL^{-1}
\biggl[\frac{\lambda^2}{\sqrt{4\nu z + |\xi|^2\nu^2}
-\lambda^2} \biggr](t)
\\
&&\qquad= \exp \biggl(-\frac{\nu t |\xi|^2}{4} \biggr) \biggl( \frac{\lambda^2}{\sqrt{4\nu\pi t}} +
\frac{\lambda^4}{2\nu} \exp \biggl( \frac{\lambda^4 t}{4\nu
} \biggr) \Phi \biggl(
\lambda^2 \sqrt{\frac{t}{2\nu}} \biggr) \biggr).
\end{eqnarray*}
Finally, take the inverse Fourier transform of the above quantity to obtain
$\calK(t,x)$ as in \eqref{E:K}, together with \eqref{E:SecMom}.
\end{remark}

Assume that $\rho:\R\mapsto\R$ is globally Lipschitz
continuous with Lipschitz constant $\LIP_\rho>0$.
We need some growth conditions on $\rho$:
Assume that for some constants $\mathit{L}_\rho>0$ and $\Vip\ge0$,
%
%e2.12 #&#
\begin{equation}
\label{E1:LinGrow} \bigl|\rho(x)\bigr|^2 \le\mathit{L}_\rho^2
\bigl(\Vip^2 +x^2 \bigr)\qquad\mbox{for all $x\in\R$}.
\end{equation}
Note that we can always take $\mathit{L}_\rho\le\sqrt{2}\LIP_\rho$, and the
inequality may even be strict.
In order to bound the second moment from below,
we will sometimes assume that for some constants $\mathit{l}_\rho>0$ and
$\underline{\varsigma}\ge0$,
%
%e2.13 #&#
\begin{equation}
\label{E1:lingrow} \bigl|\rho(x)\bigr|^2\ge\mathit{l}_{\rho}^2
\bigl(\underline{\varsigma }^2+x^2 \bigr)\qquad \mbox{for all
$x\in\R$}.
\end{equation}
We shall give special attention to the linear case (the parabolic
Anderson model): $\rho(u)=\lambda u$ with $\lambda\ne0$, which is a
special case of the following quasi-linear growth condition:
for some constant $\vv\ge0$,
%
%e2.14 #&#
\begin{equation}
\label{E1:qlinear} \bigl|\rho(x)\bigr|^2= \lambda^2 \bigl(
\vv^2+x^2 \bigr)\qquad \mbox{for all $x\in\R$}.
\end{equation}

Recall the formula for $\calK(t,x)$ in \eqref{E:K}. We will use
the following conventions:
%
%e2.15 #&#
\begin{eqnarray}
\label{E:convention} %
\calK(t,x) &:= &\calK (t,x ; \nu,
\lambda ), \qquad\overline{\calK}(t,x) := \calK (t,x ; \nu,\mathit{L}_\rho
),
\nonumber
\\
\underline{\calK}(t,x) &:=&\calK (t,x; \nu,\mathit{l}_\rho ),\qquad
\widetilde{\calK}_p(t,x) :=\calK (t,x ; \nu,a_{p,\Vip}
z_p \mathit{L}_\rho )\\
 \eqntext{\mbox{for all $p>2$}, }
\end{eqnarray}
where the constant $a_{p,\Vip}$ $(\le2)$ is defined by
%
%e2.16 #&#
\begin{eqnarray}
\label{E:a_pv} a_{p,\Vip} := \cases{ 2^{(p-1)/p},& \quad$\mbox{if $\Vip
\ne0, p>2$}$,
\cr
\sqrt{2}, & \quad $\mbox{if $\Vip=0, p>2$}$,
\cr
1, &\quad  $\mbox{if $p=2$}$,}
\end{eqnarray}
and $z_p$ is the universal constant in the
Burkholder--Davis--Gundy inequality (see~\cite{ConusKhosh10Farthest}, Theorem~1.4; in particular, $z_2=1$),
and so
%
%e2.17 #&#
\begin{equation}
\label{E:z_p} z_p\le2\sqrt{p}\qquad \mbox{for all $p\ge2$.}
\end{equation}
Note that $\widetilde{\calK}_p(t,x)$ implicitly depends on
$\Vip$ through $a_{p,\Vip}$, which will be clear from the context. If $p=2$,
then $\widetilde{\calK}_p(t,x) = \overline{\calK}(t,x)$.
For $t\ge0$, define
%
%e2.18 #&#
\begin{equation}
\label{E2:H} \calH(t;\nu,\lambda):= (1\star\calK ) (t,x) = 2
 e^{{\lambda^4t}/{(4\nu)}}
\Phi \biggl(\lambda^2\sqrt{\frac{t}{2\nu}} \biggr)-1
\end{equation}
(see Lemma~\ref{LA:Phi-E} for the second equality).
In particular, by \eqref{E:K} we can write
%
%e2.19 #&#
\begin{equation}
\label{E:K2} \calK (t,x;\nu,\lambda )=G_{{\nu}/{2}}(t,x) \biggl(
\frac{\lambda^2}{\sqrt{4\pi\nu t}}+ \frac{\lambda^4}{4\nu} \bigl(\calH(t\dvtx\nu,\lambda)+1 \bigr)
\biggr) .
\end{equation}
We also apply the conventions of \eqref{E:convention} to the kernel functions
$\calL_n (t,x;\nu,\lambda )$ and $\calH(t;\nu,\lambda)$.

Let $\cdot$ and $\circ$ denote time and space dummy variables, respectively.
For $\tau\ge t> 0$ and $x, y\in\R$, define
%
%e2.20 #&#
\begin{eqnarray}
\label{rd1}
&&\calI(t,x,\tau, y;\nu,\vv,\lambda)\nonumber\\
&&\qquad :=
\lambda^2\int_0^t \,\ud r\int
_\R\,\ud z \bigl[ J_0^2(r,z)+
\bigl(J_0^2(\cdot,\circ)\star \calK(\cdot,\circ;\nu,
\lambda) \bigr) (r,z)+ \vv^2 \calH(r;\nu,\lambda) \bigr]
\nonumber
\\
&&\qquad\quad{} \times G_\nu(t-r,x-z)G_\nu(\tau-r,y-z)
\\
\nonumber
&&\qquad\quad{} + \frac{\lambda^2 \vv^2}{\nu}|x-y| \biggl(\Phi \biggl(\frac
{|x-y|}{\sqrt
{\nu
(t+\tau)}}
\biggr)-\Phi \biggl(\frac{|x-y|}{\sqrt{\nu(\tau
-t)}} \biggr) \biggr)
\\
&&\qquad\quad{} + \lambda^2 \vv^2 \bigl[(t+\tau) G_{\nu} (t+
\tau,x-y ) - (\tau-t) G_{\nu} (\tau-t,x-y ) \bigr].\nonumber
\end{eqnarray}
When $\tau= t$ in this formula, we set $\Phi(|x-y|/0) =1$.

%th2.4 #&#
\begin{theorem}[(Existence, uniqueness and moments)]
\label{T2:ExUni}
Suppose that
the function $\rho$ is Lipschitz continuous and satisfies \eqref{E1:LinGrow},
and $\mu\in\calM_H(\R)$.
Then the stochastic integral equation \eqref{E2:WalshSI} has a random
field solution
$u=\{u(t,x), (t,x)\in\R_+^*\times\R\}$. Moreover:
\begin{longlist}[(1)]
\item[(1)] $u$ is unique (in the sense of versions).

\item[(2)] $(t,x)\mapsto u(t,x)$ is $L^p(\Omega)$-continuous for all integers
$p\ge
2$.

\item[(3)] For all even integers $p\ge2$, all $\tau\ge t>0$ and $x,y\in\R$,
%
%e2.21 #&#
\begin{equation}\qquad
\label{E2:SecMom-Up} \bigl\Vert u(t,x)\bigr\Vert_p^2 \le \cases{
J_0^2(t,x) + \bigl(J_0^2\star
\overline{\calK} \bigr) (t,x) + \Vip^2 \overline{\calH}(t),&\quad $\mbox{if
$p=2$,}$
\cr
2J_0^2(t,x) + \bigl(2J_0^2
\star\widetilde{\calK}_p \bigr) (t,x) + \Vip^2 \widetilde{
\calH}_p(t),& \quad $\mbox{if $p>2$,}$}
\end{equation}
and
%
%e2.22 #&#
\begin{equation}
\label{E2:TP-Up} \E \bigl[u(t,x)u (\tau,y ) \bigr] \le J_0(t,x)J_0
(\tau,y ) + \calI(t,x,\tau,y;\nu,\Vip ,\mathit{L}_\rho).
\end{equation}
\item[(4)] If $\rho$ satisfies \eqref{E1:lingrow}, then for all $\tau\ge
t>0$ and
$x,y\in\R$,
%
%e2.23 #&#
\begin{equation}
\label{E2:SecMom-Lower} \bigl\Vert u(t,x)\bigr\Vert_2^2 \ge
J_0^2(t,x) + \bigl(J_0^2
\star\underline {\calK} \bigr) (t,x)+ \underline{\varsigma}^2
\underline{\calH}(t) % \mbox{and}
\end{equation}
and
%
%e2.24 #&#
\begin{equation}
\label{E2:TP-Lower} \E \bigl[u(t,x)u (\tau,y ) \bigr] \ge J_0(t,x)J_0
(\tau,y ) + \calI(t,x,\tau,y;\nu,\underline {\varsigma},\mathit{l}_\rho).
\end{equation}
\item[(5)] In particular, if $|\rho(u)|^2=\lambda^2
 (\vv^2+u^2 )$, then for all $\tau\ge t>0$ and $x,y\in\R$,
%
%e2.25 #&#
\begin{equation}
\label{E2:SecMom} \bigl\Vert u(t,x)\bigr\Vert_2^2 =
J_0^2(t,x) + \bigl(J_0^2
\star\calK \bigr) (t,x)+ \vv^2 \calH(t) % \mbox{and}
\end{equation}
and
%
%e2.26 #&#
\begin{equation}
\label{E2:TP} \E \bigl[u(t,x)u (\tau,y ) \bigr] = J_0(t,x)J_0
(\tau,y ) + \calI(t,x,\tau,y;\nu,\vv ,\lambda).
\end{equation}
\end{longlist}
\end{theorem}

This theorem will be proved in Section~\ref{SS2:Existence}.
We note that it is not clear if~\eqref{E2:SecMom-Up} holds when $p>2$
is a real
number but \emph{not} an even integer. However, if $k\in\{2,3,\ldots
\}$ and
$2(k-1)<p\le2k$, then $\Vert u(t,x)\Vert_p^2\le\Vert u(t,x)\Vert
_{2k}^2$ and
\eqref{E2:SecMom-Up} applies to $\Vert u(t,x)\Vert_{2k}^2$.

%co2.5 #&#
\begin{corollary}[(Constant initial data)] \label{C2:TP-Lebesgue}
Suppose that $|\rho(u)|^2=\lambda^2(\vv^2+u^2)$ and $\mu$ is
Lebesgue measure. Then for all $\tau\ge t>0$ and $x,y\in\R$,
%
%e2.27 #&#
\begin{eqnarray}\quad
\label{E2:TP-Lebesgue}&& \E \bigl[u(t,x)u (\tau,y ) \bigr]
\nonumber
\\
&&\qquad=1+\bigl(1+\vv^2\bigr) \\
&&\qquad\quad{}\times\biggl[\exp \biggl(\frac{\lambda^4 \bar{t}-2 \lambda^2 |x-y|}{4
\nu} \biggr)
\Erfc \biggl(\frac{|x-y|-\lambda^2 \bar{t}}{2
(\nu\bar{t} )^{1/2}} \biggr) - \Erfc \biggl(\frac{|x-y|}{2 (\nu\bar{t} )^{1/2}} \biggr)
\biggr],\nonumber % \mbox{and}\\
\end{eqnarray}
where $\bar{t} = (t+\tau)/2$, and
%
%e2.28 #&#
\begin{equation}
\label{E2:SecMom-Lebesgue} \E \bigl[\bigl|u(t,x)\bigr|^2 \bigr] = 1+\bigl(1+
\vv^2\bigr)\calH(t).
\end{equation}
\end{corollary}

\begin{pf}
In this case, $J_0(t,x)\equiv1$.
Formula \eqref{E2:SecMom-Lebesgue} follows from \eqref{E2:SecMom} and~\eqref{E2:H}.
By \eqref{E2:TP} and using Lemma~\ref{L:IntIntGG} to account for the
last two terms in \eqref{rd1}, we see that
\begin{eqnarray*}
\E \bigl[u(t,x)u (\tau,y ) \bigr] & =&1+\lambda^2 \int
_0^t \,\ud r \int_\R\,\ud z
\bigl[\vv^2+1+\bigl(1+\vv^2\bigr)\calH (r) \bigr]\nonumber\\
&&{}\times G_\nu(t-r,x-z)G_\nu (\tau-r,y-z )
\\
&=&1+\lambda^2\bigl(1+\vv^2\bigr) \int
_0^t \bigl(\calH(r)+1 \bigr) G_{2\nu}
\biggl(\frac{t+\tau}{2}-r,x-y \biggr)\,\ud r,\nonumber
\end{eqnarray*}
and this last integral is evaluated by Lemma~\ref{L2:IntHG}.
\end{pf}

%re2.6 #&#
\begin{remark}\label{R2:TP-Lebesgue}
If $\rho(u)=u$ (i.e., $\lambda=1$ and $\vv=0$), then
\eqref{E2:SecMom-Lebesgue} recovers, in the case
$n=2$, the moment formulas of
Bertini and Cancrini
\cite{BertiniCancrini94Intermittence}, Theorem~2.6.
As for the two-point correlation function, \cite
{BertiniCancrini94Intermittence}, Corollary~2.4, states the integral formula
%
%e2.29 #&#
\begin{eqnarray}
\label{E2:TP-Lebesgue-BC}&& \E \bigl[u(t,x)u (t,y ) \bigr]
\nonumber
\\[-8pt]
\\[-8pt]
\nonumber
&&\qquad = \int_0^t
\,\ud s \frac{|x-y|}{\sqrt{\pi\nu s^3}} \exp \biggl\{-\frac{(x-y)^2}{4\nu
s}+\frac{t-s}{4\nu}
\biggr\}\Phi \biggl(\sqrt{\frac{t-s}{2\nu
}} \biggr).
\end{eqnarray}
By Lemma~\ref{L2:IntHG-BC} below, the integral is equal to
\[
e^{{(t-2 |x-y|)}/{(4 \nu)}} \Erfc \bigl((4\nu t)^{-1/2}\bigl(|x-y|-t\bigr) \bigr),
\]
so their result differs from ours. The difference is a term
\[
1-\Erfc \bigl((4\nu t)^{-1/2}|x-y| \bigr) =\Erf \bigl((4\nu
t)^{-1/2}|x-y| \bigr),
\]
which vanishes when $x=y$.
However, for $x\ne y$, this is not the case. For instance,
as $t$ tends to zero, the correlation
function should have a limit equal to one, while~\eqref{E2:TP-Lebesgue-BC}
has limit zero.
The argument in \cite{BertiniCancrini94Intermittence} should be
modified as
follows (we use the notation in
their paper): (4.6) on page 1398 should be
\[
\E_{0}^{\beta,1} \biggl[ \exp \biggl(\frac{L_t^{\xi}(\beta)}{\sqrt{2\nu}} \biggr)
\biggr] = \int_0^t P_\xi(\ud s)
\E_0^\beta \biggl[\exp \biggl(\frac{L_{t-s}(\beta)}{\sqrt{2\nu
}} \biggr)
\biggr] + P(T_\xi\ge t).
\]
The extra term $P(T_\xi\ge t)$ is equal to
\[
\int_t^\infty \frac{|\xi|}{\sqrt{2\pi s^3}} \exp \biggl(-
\frac{\xi^2}{2s} \biggr) \,\ud s = \Erf \biggl(\frac{|\xi|}{\sqrt{2t}} \biggr) = \Erf
\biggl(\frac{\llvert x-x'\rrvert }{\sqrt{4\nu t}} \biggr).
\]
With this term, \eqref{E2:TP-Lebesgue} is recovered.
% \myEnd
\end{remark}

%ex2.7 #&#
\begin{example}[(Higher moments for constant initial data)]
\label{Ex2:MomLeb}
Suppose that $\mu(\ud x) = \,\ud x$. Then $J_0(t,x)\equiv1$.
By \eqref{E2:SecMom-Up},
\[
\E\bigl[\bigl|u(t,x)\bigr|^p\bigr] \le2^{p-1} + 2^{p-1}
\bigl(2+\Vip^2 \bigr)^{p/2} \exp \biggl(\frac{a_{p,\Vip}^4 z_p^4 p \mathit{L}_\rho^4
t}{8\nu}
\biggr).
\]
Using \eqref{E:z_p} and \eqref{E:a_pv}, replace $z_p$ by $2\sqrt{p}$,
and $a_{p,\Vip}$ by $2$. Thus,
$\overline{m}_p(x)\equiv\overline{m}_p \le
2^5 p^3 \mathit{L}_\rho^4/\nu$.
If $\Vip=0$, we can replace $a_{p,\Vip}$ by $\sqrt{2}$ instead of
$2$, which
gives a slightly better bound: $\overline{m}_p\le2^3 p^3
\mathit{L}_\rho^4/\nu$.
In particular, for the parabolic Anderson model $\rho(u) = \lambda u$,
we obtain
$\overline{m}_p\le2^3 p^3 \lambda^4/\nu$, which is consistent with
Bertini and Cancrini's formula:
$m_p= \frac{\lambda^4}{4! \nu} p(p^2-1)$ (see
\cite{BertiniCancrini94Intermittence}, (2.40)).
% \myEnd
\end{example}

%co2.8 #&#
\begin{corollary}[(Dirac delta initial data)]\label{C2:TP-Delta}
Suppose that $|\rho(u)|^2=\lambda^2(\vv^2+u^2)$ and $\mu$ is the
Dirac delta
measure with a unit mass at zero. Then for all $t>0$ and $x,y\in\R$,
%
%e2.30 #&#
\begin{eqnarray}
\label{E2:TP-Delta} \E \bigl[u(t,x)u (t,y ) \bigr] &= &G_\nu(t,x)G_\nu
(t,y ) -\vv^2 \Erfc \biggl(\frac{|x-y|}{2 \sqrt{\nu
t}} \biggr)\nonumber
\\
&&{}+ \biggl( \frac{\lambda^2}{4\nu} G_{{\nu}/{2}} \biggl(t,\frac{x+y}{2}
\biggr)+\vv^2 \biggr) \exp \biggl(\frac{\lambda^4 t-2 \lambda^2 |x-y|}{4 \nu}
\biggr) \\
&&{}\times\Erfc
\biggl(\frac{|x-y|-\lambda^2 t}{2
\sqrt{\nu t}} \biggr)
\nonumber
\end{eqnarray}
and
%
%e2.31 #&#
\begin{equation}
\label{E2:SecMom-Delta} \E \bigl[\bigl|u(t,x)\bigr|^2 \bigr] = \frac{1}{\lambda^2}
\calK(t,x)+\vv^2\calH(t).
\end{equation}
\end{corollary}

This corollary is proved in Section~\ref{Ss:EUM-OtherProof}.

%re2.9 #&#
\begin{remark}\label{R2:TP-Delta}
If $\rho(u)=u$ (i.e., $\lambda=1$ and $\vv=0$), then
\eqref{E2:SecMom-Delta} coincides with the result
by Bertini and Cancrini
\cite{BertiniCancrini94Intermittence}, (2.27) (see also
\cite{BorodinCorwin11MP,AmirCorwinQuastel11Stationary}):
$\E [|u(t,x)|^2 ] =\calK (t,x )$.
As for the two-point correlation function, Bertini and Cancrini gave the
following
integral (see \cite{BertiniCancrini94Intermittence}, Corollary~2.5):
%
%e2.32 #&#
\begin{eqnarray}
\label{E2:TP-Delta-BC} \E \bigl[u(t,x)u (t,y ) \bigr] &=& \frac{1}{2\pi\nu
t}\exp \biggl\{-
\frac{x^2+y^2}{2\nu t} \biggr\} \int_0^1 \,\ud s
\frac{|x-y|}{\sqrt{4\pi\nu t}} \frac{1}{\sqrt{s^3(1-s)}}
\nonumber\\
&&{}\times\exp \biggl\{-\frac{(x-y)^2}{4\nu t}\frac{1-s}{s} \biggr\} \\
&&{}\times\biggl(1+
\sqrt{\frac{\pi t (1-s)}{\nu}} \exp \biggl\{\frac{t}{2\nu} \frac{1-s}{2}
\biggr\} \Phi \biggl(\sqrt{\frac{t(1-s)}{2\nu}} \biggr) \biggr).
\nonumber
\end{eqnarray}
This integral can be evaluated explicitly (see Lemma~\ref{L2:TP-Delta-BC}
below) and coincides with \eqref{E2:TP-Delta} for $\vv=0$ and
$\lambda=1$.
\end{remark}

%ex2.10 #&#
\begin{example}[(Higher moments for delta initial data)]
\label{Ex:GrowthDeta1}
Suppose that $\mu= \delta_0$ and $\Vip=0$. Let $p\ge2$ be an even
integer. Clearly, $J_0(t,x)\equiv G_\nu(t,x)$. Then by~\eqref{E2:SecMom-Up}
and \eqref{E:KS-Kn-K0},
\[
\E \bigl[ \bigl|u(t,x)\bigr|^p \bigr] \le2^{p-1} G_\nu^{p}(t,x)+2^{(p-2)/2}
\mathit{L}_\rho^{-p} z_p^{-p} \bigl
\llvert \widetilde{\calK}_p(t,x)\bigr\rrvert ^{p/2}.
\]
It follows from \eqref{E:K} and \eqref{E:z_p} that for all $x\in\R$,
$\overline{m}_p(x)\le\mathit{L}_\rho^4 z_p^4 p/(2\nu)
\le2^3 p^3 \mathit{L}_\rho^4 /\nu$.
Note that this upper bound is identical to the case of
the constant initial data (Example~\ref{Ex2:MomLeb}).
Concerning the growth indices, we see from \eqref{E:K} that
\[
\lim_{t\rightarrow+\infty}\frac{1}{t} \sup_{|x|> \alpha
t}\log
\E \bigl[\bigl|u(t,x)\bigr|^p \bigr] \le-\frac{\alpha^2 p}{2\nu} + \frac{\mathit{L}_\rho^4 p
z_p^4}{2\nu}\qquad \mbox{for all $\alpha\ge0$}.
\]
Hence, $\overline{\lambda}(p) \le z_p^2 \mathit{L}_\rho^2$. Similarly,
$\underline{\lambda}(2)\ge\mathit{l}_\rho^2/2$ after using \eqref
{E2:SecMom-Lower}.
Therefore, $\frac{\mathit{l}_\rho^2}{2} \le
\underline{\lambda}(p) \le\overline{\lambda}(p)
\le z_p^2 \mathit{L}_\rho^2$ for all even integers $p\ge2$.
The same bounds are obtained for more general
initial data in Theorem~\ref{T2:Growth}.
% \myEnd
\end{example}

The following proposition shows
that initial data cannot be extended beyond
measures.

%pr2.11 #&#
\begin{proposition}\label{P2:D-Delta}
Suppose that $\mu= \delta_0^{\prime}$ (the derivative of the
Dirac delta measure at zero). Let $\rho(u) =
\lambda u$ ($\lambda\ne0$). Then \eqref{E2:WalshSI}
does not have a random field solution.
\end{proposition}

The proof of this proposition is given in Section~\ref{Ss:EUM-OtherProof}.

%s2.2 #&#
\subsection{Growth indices}\label{SS:Main-Exp}
For $\sd\ge0$, define
\[
\calM_G^{\sd} (\R ):= \biggl\{\mu\in\calM (\R )\dvtx\int
_\R e^{\sd|x|}|\mu| (\ud x)<+\infty \biggr\}.
\]
Let $\calM_+ (\R )$ denote the set of nonnegative Borel measures
over $\R$,
\[
\calM_{G,+}^{\sd} (\R ) = \calM_G^{\sd}
(\R ) \cap \calM_+ (\R ) \quad\mbox{and}\quad \calM_{H,+} (\R ) =
\calM_H (\R ) \cap \calM_+ (\R ).
\]
Recall the definitions of $\underline{\lambda}(p)$
and $\overline{\lambda}(p)$ in \eqref{E1:GrowInd-0} and \eqref
{E1:GrowInd-1}.

%th2.12 #&#
\begin{theorem}\label{T2:Growth}
(1) Suppose that $|\rho(u)|^2\ge\mathit{l}_\rho^2 (\underline
{\varsigma}^2+u^2 )$ and
$p\ge2$.
If $\underline{\varsigma}=0$, then
$\underline{\lambda}(p)\ge\mathit{l}_\rho^2/2$ for all $\mu\in
\calM_{H,+} (\R )$ with $\mu\ne0$;
if $\underline{\varsigma}\ne0$, then $\underline{\lambda}(p) =
\overline{\lambda}(p)=+\infty$, for all $\mu\in
\calM_{H,+} (\R )$.

(2) If $|\rho(u)|^2\le\mathit{L}_\rho^2 (\Vip^2+u^2 )$
with $\Vip=0$
(which implies $\underline{\varsigma}= \vv= 0$) and $\mu\in
\calM_{G}^{\sd} (\R )$ for some $\sd> 0$, then for all even
integers $p\ge2$,
\[
\overline{\lambda}(p)\le
\cases{ \displaystyle\frac{\sd\nu}{2} +\frac{z_{p}^4\mathit{L}_\rho^4}{2\nu\sd},
 &\quad $
\mbox{if } 0\le\sd< \nu^{-1}z_{p}^2
\mathit{L}_\rho^2$,
\cr
z_{p}^2
\mathit{L}_\rho^2 , &\quad $\mbox{if } \sd\ge
\nu^{-1}z_{p}^2\mathit{L}_\rho^2$.}
\]
In addition,
\[
\overline{\lambda}(2)\le
 \cases{\displaystyle \frac{\sd\nu}{2} +\frac{\mathit{L}_\rho^4}{8\nu\sd} , &\quad $
\mbox{if }\displaystyle 0\le\sd< \frac{\mathit{L}_\rho^2}{2\nu} $,
\cr
\displaystyle\frac{1}{2}
\mathit{L}_\rho^2 , &\quad $\mbox{if }\displaystyle \sd\ge\frac{\mathit{L}_\rho^2}{2\nu}
$.}
\]

(3) Suppose that $|\rho(u)|^2=\lambda^2 (\vv^2+u^2 )$,
$\lambda
\ne
0$. If $\vv=0$ and $\sd\ge\frac{\lambda^2}{2\nu}$, then
$\underline{\lambda}(2)=\overline{\lambda}(2)=\lambda^2/2$
for all
$\mu\in\calM_{G,+}^{\sd} (\R )$ with $\mu\ne0$; if
$\vv\ne
0$, then
$\underline{\lambda}(p) =
\overline{\lambda}(p)=+\infty$ for all $\mu\in
\calM_{H,+} (\R )$ and $p\ge2$.
\end{theorem}

This theorem generalizes the results in \cite{ConusKhosh10Farthest} in several
regards:
(i) more general initial data are allowed; (ii) both nontrivial upper bounds
and lower bounds are given (compare with \cite{ConusKhosh10Farthest}, Theorem~1.1)
for the Laplace operator case; (iii) for the parabolic Anderson model,
the exact transition is proved (see Theorem~1.3 and the first open
problem in
\cite{ConusKhosh10Farthest}) for $n=2$ and the Laplace operator case; (iv)
our discussions above cover the case where $\rho(0)\ne0$.
The lower bounds are proved in Section~\ref{SS2:ExpInd-Low}, the
upper bounds in Section~\ref{SS2:ExpInd-Up}.

%ex2.13 #&#
\begin{example}[(Delta initial data)]
Suppose that $\Vip=\underline{\varsigma}=0$. Clearly, $\delta_0 \in
\calM_{G,+}^{\sd} (\R )$ for
all $\sd\ge0$. Hence, the above
theorem implies that for all even integers $k\ge2$,
$\frac{\mathit{l}_\rho^2}{2} \le\underline{\lambda}(k) \le
\overline{\lambda
}(k) \le
z_k^2 \mathit{L}_\rho^2$, which recovers the bounds in Example~\ref{Ex:GrowthDeta1}.
\end{example}

%pr2.14 #&#
\begin{proposition}\label{P2:Ex-Exp}
Consider the parabolic Anderson model $\rho(u) =\lambda u$, \mbox{$\lambda
\ne
0$}, with
the initial data $\mu(\ud x) = e^{-\sd|x|}\,\ud x$ ($\sd>0$). Then
%
%e2.33 #&#
\begin{eqnarray}
\label{E2:Ex-Exp} \underline{\lambda}(2)=\overline{\lambda}(2) = \cases{
\displaystyle\frac{\sd\nu}{2} + \frac{\lambda^4}{8 \sd\nu},& \quad $\mbox{if $ 0<\sd\le
\displaystyle\frac{\lambda^2}{2 \nu}$}$,\vspace*{2pt}
\cr
\displaystyle\frac{\lambda^2}{2}, &\quad $\mbox{if $ \sd\ge
\displaystyle\frac{\lambda^2}{2 \nu}$}$.}
\end{eqnarray}
\end{proposition}

This proposition shows that for all $\sd\in
 \,]0,+\infty ]$, the exact phase transition occurs,
and hence our upper bounds
for $\overline{\lambda}(2)$ in Theorem~\ref{T2:Growth} are sharp.
See Section~\ref{SS2:ExpInd-Spe} for the proof.

%s3 #&#
\section{Proof of existence, uniqueness and moment estimates}
\label{S:Proof-ExUnMm}

%s3.1 #&#
\subsection{Some criteria for predictable random fields}\label{SS2:ProbSpace}
A random field $\{Z(t,x)\}$ is called \emph{elementary} if we can write
$Z(t,x)=Y 1_{]a,b]}(t) 1_{A}(x)$, where $0\le a<b$, $A\subset\R$ is an
interval, and $Y$ is an $\calF_a$-measurable random variable. A
\textit{simple}
process is a finite sum of elementary random fields. The set of simple processes
generates the \emph{predictable} $\sigma$-field on $\R_+\times\R
\times
\Omega$,
denoted by $\calP$. For $p\ge2$ and
$X\in L^2 (\R_+\times\R, L^p(\Omega) )$, set
%
%e3.1 #&#
\begin{equation}
\label{E2:Mp-norm} \Vert X\Vert_{M, p}^2  : =
\iint_{\R_+^*\times\R}\bigl\Vert X (s,y )\bigr\Vert_p^2\,\ud s\,\ud y<+
\infty.
\end{equation}
When $p=2$, we write $\Vert X\Vert_M$ instead of $\Vert X\Vert_{M,2}$.
In \cite{Walsh86}, $\iint X\,\ud W$ is defined for predictable $X$
such that $\Vert X\Vert_M<+\infty$. However, the condition of
predictability is not
always so easy to check, and as in the case of ordinary Brownian motion~\cite{ChungWilliams90}, Chapter~3, it is convenient to be able to integrate
elements $X$ that are merely jointly measurable and adapted. For this, let
$\calP_p$
denote the closure in $L^2 (\R_+\times\R, L^p(\Omega) )$
of simple
processes.
Clearly, $\calP_2\supseteq\calP_p \supseteq\calP_q$ for $2\le p\le
q<+\infty$,
and according to \Itos isometry, $\iint X\,\ud W$ is well defined for all
elements of $\calP_2$. The next proposition gives easily verifiable conditions
for checking that $X\in\calP_2$.

%pr3.1 #&#
\begin{proposition}\label{P2:Pm-Ext}
Suppose that for some $t>0$ and $p\in[2,+\infty[ $, a random field
$X= \{ X (s,y ),  (s,y )\in\,]0,t[\times
\R
 \}$ has the following properties:
\begin{longlist}[(iii)]
\item[(i)] $X$ is adapted, that is, for all $ (s,y )\in
\,]0,t[\times
\R$,
$X (s,y )$ is $\calF_s$-measurable;
\item[(ii)] $X$ is jointly measurable with respect to
$\calB( ]0,t[ \times\R)\times\calF$;
\item[(iii)] $\Vert X(\cdot,\circ)1_{]0,t[}(\cdot )\Vert
_{M,p}<+\infty$.
\end{longlist}
Then $X(\cdot,\circ) 1_{]0,t[}(\cdot)$ belongs to $\calP_2$.
\end{proposition}

\begin{pf}
\emph{Step} 1. We first prove this proposition with (ii) replaced by:
\begin{longlist}[(ii$'$)]
\item[(ii$'$)] For all $ (s,y )\in\,]0,t[\times\R$,
$\Vert X(s,y)\Vert_p<+\infty$ and the function $ (s,y
)\mapsto
X (s,y )$ from
$]0,t[ \times\R$ into $L^p(\Omega)$ is continuous.
\end{longlist}
Fix $\varepsilon>0$ with $\varepsilon\le t/3$. Since
$\Vert X(\cdot,\circ)1_{]0,t[}(\cdot)\Vert_{M,p}<+\infty$,
choose $a=a(\varepsilon)>\max(t,2/t)$ large enough so that
\[
\iint_{ ([1/a,t-1/a] \times[-a,a] )^c} \bigl\Vert X (s,y )\bigr\Vert_p^2
1_{]0,t[}(s) \,\ud s \,\ud y<\varepsilon.
\]
Due to the $L^p(\Omega)$-continuity hypothesis in (ii$'$), we can choose
$n\in\bbN$ large enough so that for all
$(s_1,y_1), (s_2,y_2)\in[\varepsilon,t-\varepsilon]\times[-a,a]$,
\[
\max \bigl\{|s_1-s_2|,|y_1-y_2| \bigr\}
\le\frac{t-2/a}{n}\quad \Rightarrow\quad\bigl \Vert X(s_1,y_1)-X(s_2,y_2)
\bigr\Vert_p <\frac{\varepsilon}{a}.
\]
Choose $m\in\bbN$ large enough so that $a/m\le(t-2/a)/n$.
Set $t_j= \frac{j(t-2/a)}{n}+\frac{1}{a}$ with $j\in\{0,\ldots,n\}$ and
$x_i=\frac{i a}{m}-a$ with $i\in\{0,\ldots,2m\}$.
Then define
\[
X_{n,m}(t,x) = \sum_{j=0}^{n-1}
\sum_{i=0}^{2m-1} X(t_j,x_i)
1_{]t_j,t_{j+1}]}(t)1_{]x_i,x_{i+1}]}(x).
\]
Since $X$ is adapted, $X(t_j,x_i)$ is $\calF_{t_j}$-measurable,
and so $X_{n,m}$ is predictable, and clearly, $X_{n,m}\in\calP_p$.
Since $X_{n,m}(t,x)$ vanishes outside of the rectangle $[1/a,t-1/a]
\times
[-a,a]$, we have
\begin{eqnarray*}
\bigl\Vert X 1_{]0,t[}-X_{n,m}\bigr\Vert_{M, p}^2
&=& \iint_{ ([1/a,t-1/a] \times[-a,a] )^c} \bigl\Vert X (s,y )\bigr\Vert_p^2
1_{]0,t[}(s) \,\ud s \,\ud y
\\
&&{}+ \sum_{j=0}^{n-1} \sum
_{i=0}^{2m-1} \int_{t_j}^{t_{j+1}}
\int_{x_i}^{x_{i+1}}\bigl\Vert X(t_j,
x_i)-X (s,y )\bigr\Vert_p^2 \,\ud s \,\ud y
\\
&\le& \varepsilon+ \sum_{j=0}^{n-1} \sum
_{i=0}^{2m-1} \int_{t_j}^{t_{j+1}}
\int_{x_i}^{x_{i+1}} \frac{\varepsilon^2}{a^2} \,\ud s \,\ud y
\\
&=&\varepsilon+ \varepsilon^2 \frac{2at-4}{a^2} \le\varepsilon+
\frac{2\varepsilon^2 t}{a} \le\varepsilon+ 2\varepsilon^2.
\end{eqnarray*}
Therefore, $X(\cdot,\circ)1_{]0,t[}(\cdot)\in\calP_p\subseteq
\calP_2$.

\textit{Step} 2. Now we prove this proposition under (ii), assuming that
$X$ is
bounded.
Take a $\psi\in C_c^\infty(\R^2)$, nonnegative, such that
$\operatorname{supp} (\psi )\subset\,]0,t[\times\,]{-}1,1[$
and $\iint_{\R^2}\psi (s,y )\,\ud s\,\ud y =1$.
Let\vspace*{1pt} $\psi_n (s,y ) := n^2\psi(n s,n y)$ for each\vadjust{\goodbreak} $n\in
\bbN^*$, and
$\widetilde{X}_n (s,y ) := (\psi_n\star X
)
(s,y )$
for all $ (s,y )\in\,]0,t[\times\R$.
Note that when we do the convolution in time, $X (s,y )$ is
understood to be zero for $s\notin\,]0,t[$.

We shall first prove that
$\widetilde{X}_n(\cdot,\circ)1_{]0,t[}(\cdot)\in\calP_{2}$
for all $n\in\bbN^*$
and
%
%e3.2 #&#
\begin{equation}
\label{E2:Pm-Ext} \bigl\Vert\widetilde{X}_n(\cdot,\circ) 1_{]0,t[}
\bigr\Vert_{M,2} \le \bigl\Vert X(\cdot,\circ) 1_{]0,t[}
\bigr\Vert_{M,2}<+\infty.
\end{equation}
The inequality \eqref{E2:Pm-Ext} is true since, by H\"{o}lder's inequality,
\[
\bigl\Vert\widetilde{X}_n(\cdot,\circ)1_{]0,t[}(\cdot)
\bigr\Vert_{M,2}^2 \le \iint_{[0,t]\times\R}\,\ud s\,\ud y
\iint_{\R^2}\E \bigl(X^2(u,z) \bigr)\psi_n(s-u,y-z)
\,\ud u\,\ud z,
\]
which is less than $\Vert X(\cdot,\circ)1_{]0,t[}(\cdot)\Vert
_{M,2}^2$ and
is finite
by property (iii).

The condition that
$\operatorname{supp} (\psi )\subset\R_+^*\times\R$,
together with the joint
measurability of $X$, ensures that $\widetilde{X}_n$ is still adapted.
The sample path continuity of $\widetilde{X}_n$ in both the space and time
variables implies $L^2(\Omega)$-continuity, thanks to the boundedness of
$X$.
Hence, we can apply step 1 to conclude that
$\widetilde{X}_n(\cdot,\circ)1_{]0,t[}(\cdot)\in\calP_{2}$, for all
$n\in\bbN^*$.

Property (iii) implies that there is $\Omega'\subseteq\Omega$ such that
$P(\Omega')=1$ and for all $\omega\in\Omega'$, $X(\cdot,\circ
,\omega)\in
L^2(]0,t[\times\R)$. Now fix $\omega\in\Omega'$. Then
\[
\lim_{n\rightarrow+\infty} \bigl\Vert\widetilde{X}_n(\cdot,\circ,
\omega)-X(\cdot,\circ,\omega )\bigr\Vert_{L^2(]0,t[
\times\R)}= 0
\]
and
\[
\bigl\Vert\widetilde{X}_n(\cdot,\circ,\omega)\bigr\Vert_{L^2(]0,t[\times\R)} \le
\bigl\Vert X(\cdot,\circ,\omega)\bigr\Vert_{L^2(]0,t[\times\R)}
\]
(see, e.g., \cite{Adam03SecondEd}, Theorem~2.29(c)).
Thus, by Lebesgue's dominated convergence theorem, which applies by (iii),
\[
\lim_{n\rightarrow\infty}\E \bigl[\bigl\Vert\widetilde{X}_n(\cdot ,
\circ) -X(\cdot,\circ)\bigr\Vert_{L^2(]0,t[\times\R)}^2 \bigr] =0.
\]
We conclude that $X(\cdot,\circ)1_{]0,t[}(\cdot)\in\calP_{2}$.

\textit{Step} 3.
Now we consider a general $X$ satisfying (i), (ii) and (iii). For
$M>0$, denote
\[
X^M(s,y,\omega)1_{]0,t[}(s) = \cases{ X(s,y,\omega)
1_{]0,t[}(s), &\quad $\mbox{if $\bigl\llvert X(s,y,\omega)\bigr\rrvert \le M$,}$
\cr
0,&\quad $\mbox{otherwise.}$}
\]
Since each $X^M(\cdot,\circ)1_{]0,t[}(\cdot)$ is bounded, satisfies
(i), (ii) and (iii), and\break 
$X^M(\cdot,\circ)1_{]0,t[}(\cdot) \rightarrow
X(\cdot,\circ)1_{]0,t[}(\cdot)$ in
$\Vert\cdot\Vert_{M,2}$ as $M\rightarrow+\infty$
(by Lebesgue's dominated convergence theorem),
we conclude from step 2 that $X(\cdot,\circ)1_{]0,t[}(\cdot)\in
\calP_{2}$.
\end{pf}

%re3.2 #&#
\begin{remark}
The step 1 in the proof of Proposition~\ref{P2:Pm-Ext} is an extension (but
specialized to space--time white noise) of Dalang and Frangos's result in
\cite{DalangFrangos98}, Proposition~2, since the second moment of
$X$ can explode at $s=0$ or $s=t$.
\end{remark}

%s3.2 #&#
\subsection{$L^p$-bounds on stochastic convolutions}
We will need an extension of \cite{ConusKhosh10Farthest}, Lemma~2.4, to
allow all
adapted, jointly measurable and integrable random fields (see also
\cite{FoondunKhoshnevisan08Intermittence}, Lemma~3.4).

%le3.3 #&#
\begin{lemma}%(\cite[Lemma~2.4]{ConusKhosh10Farthest})
\label{L2:Lp}
Let $\calG(s,y)$ be a deterministic measurable function from $\R_+^*
\times\R$
to $\R$ and let $Z= (Z (s,y ),
 (s,y )\in\R_+^* \times\R )$ be a process with the
following properties:
\begin{longlist}[(1)]
\item[(1)]$Z$ is adapted and jointly measurable with respect to
$\calB(\R_+^*\times\R)\times\calF$;
\item[(2)]$\E [\iint_{[0,t]\times\R} \calG^2 (t-s,x-y )
Z^2 (s,y )\,\ud
s\,\ud
y ]<\infty$,
for all $(t,x)\in\R_+\times\R$.
\end{longlist}
%
%Thanks to Proposition~\ref{P2:Pm-Ext},
Then for each $(t,x)\in\R_+\times\R$,
the random field $ (s,y )\in\,]0,t[ \times\R\mapsto
\calG (t-s,x-y ) Z (s,y )$ belongs to $\calP_2$
and so
the stochastic convolution
%
%e3.3 #&#
\begin{equation}
\label{E2:StoCon} (\calG\star Z\dot{W} ) (t,x) := \iint_{[0,t]\times\R} \calG
(t-s,x-y )Z (s,y )W (\ud s,\ud y )
\end{equation}
is a well-defined Walsh integral and the random field $\calG\star Z\W
$ is
adapted.
%to $\{\calF_s, s\ge0\}$.
%Let $Z$ be a random field that satisfies the above two properties. Then
Moreover, for all even integers $p\ge2$ and $(t,x)\in\R_+ \times\R$,
\[
\bigl\Vert (\calG\star Z\dot{W} ) (t,x)\bigr\Vert_p^2 \le
z_{p}^{2} \bigl\Vert\calG(t-\cdot,x-\circ) Z(\cdot,\circ)
\bigr\Vert_{M,p}^2.
\]
\end{lemma}

We note that \cite{ConusKhosh10Farthest} assumes that $Z$ is predictable.
However, using Proposition~\ref{P2:Pm-Ext}, the proof of this lemma is
the same
as that of \cite{ConusKhosh10Farthest}.

% \subsection{$L^p$-bounds on Stochastic Convolutions}
%\label{SS2:aproposition}

%pr3.4 #&#
\begin{proposition}\label{P2:Picard}
Suppose that for some even integer $p\in[2,+\infty[ $, a random field
$Y= (Y(t,x),
(t,x)\in\R_+^*\times\R )$
has the following three properties:
\begin{longlist}[(iii)]
\item[(i)]$Y$ is adapted;
\item[(ii)]$Y$ is jointly measurable with respect to
$\calB (\R_+^*\times\R )\times\calF$;
\item[(iii)] for all $(t,x)\in\R_+^*\times\R$,
$\Vert G_\nu(t-\cdot,x-\circ)Y(\cdot,\circ)\Vert_{M,p}^2
<+\infty$.
\end{longlist}
Then for all $(t,x)\in\R_+^*\times\R$,
$G_\nu(t-\cdot,x-\circ)Y(\cdot,\circ)\in\calP_2$ and the random field
\[
w(t,x)=\iint_{]0,t[\times\R} G_\nu (t-s,x-y )Y (s,y ) W(\ud s,\ud y)
\]
has the property that if $Y$ has locally bounded $p$th moments, that
is, for
$K\subset
\R_+^*\times\R$ compact,
%
%e3.4 #&#
\begin{equation}
\label{E2:Ytxlb} \sup_{(t,x)\in K}\bigl\Vert Y(t,x)\bigr\Vert_{p}<+
\infty,
\end{equation}
which is the case if $Y$ is $L^p(\Omega)$-continuous, then $w$ is
$L^p(\Omega)$-continuous on \mbox{$\R_+^*\times\R$}.
\end{proposition}

Before proving this proposition, we need the following proposition.

%pr3.5 #&#
\begin{proposition}\label{P2:G}
There are three universal constants $C_1=1$,
%$C_1=2+ \sqrt{2/\pi}\sup_{x\in\R} \frac{1-e^{-x^2/2}}{x} \approx
%2.360$,
$C_2 =\frac{\sqrt{2}-1}{\sqrt{\pi}}$, and
$C_3 = \frac{1}{\sqrt{\pi}}$,
such that for all $s,t$ with $0\le s\le t$ and $x\in\R$,
%
%e3.5 #&#
%e3.6 #&#
%e3.7 #&#
\begin{eqnarray}
\label{E2:G-x} &&\int_0^t\,\ud r\int
_\R\,\ud z \bigl[G_\nu(t-r,x-z)-G_\nu
(t-r,y-z) \bigr]^2 \le\frac{C_1}{\nu} |x-y| ,
\\
\label{E2:G-t1}&& \int_0^s\,\ud r\int
_\R\,\ud z \bigl[G_\nu(t-r,x-z)-G_\nu
(s-r,x-z) \bigr]^2 \le\frac{C_2}{\sqrt{\nu}} \sqrt{t-s} ,
\\
\label{E2:G-t2}&& \int_s^t\,\ud r\int
_\R\,\ud z \bigl[G_\nu(t-r,x-z)
\bigr]^2 \le\frac{C_3}{\sqrt{\nu}} \sqrt{t-s} ,
\\
\nonumber
&&\iint_{\R_+\times\R} \bigl(G_\nu(t-r,x-z)-G_\nu(s-r,y-z)
\bigr)^2 \,\ud r \,\ud z \\
&&\qquad\le 2C_1 \biggl(\frac{|x-y|}{\nu} +
\frac{\sqrt{|t-s|}}{\sqrt{\nu
}} \biggr) ,\nonumber
\end{eqnarray}
where we use the convention that $G_\nu(t,\cdot)\equiv0$ if $t\le0$.
\end{proposition}

%re3.6 #&#
\begin{remark}
Similar estimates can be found in, for example,
\cite{Shiga94Two}, Lem\-ma~6.2, and
\cite{Khoshnevisan09Mini}, Theorem~6.7. The above is a slight improvement because all three
constants are best possible.
Since the values of these constants are not essential here,
we refer to \cite{LeChen13Thesis}, Proposition~2.3.9,
for the proof. Note that $C_1 = 1$ was not obtained in this reference,
but with
a slight change in the last lines of the proof of~\cite
{LeChen13Thesis}, Proposition~2.3.9(i), the value $C_1 = 1$ can be obtained, and
this is
optimal.
\end{remark}

\begin{pf*}{Proof of Proposition~\ref{P2:Picard}}
Fix $(t,x)\in\R_+^*\times\R$.
Clearly,
$X= (X(s,y), \break   (s,y )\in\,]0,t[ \times\R )$
with
$X (s,y )=Y (s,y )G_\nu (t-s,x-y )$
satisfies all conditions of Proposition~\ref{P2:Pm-Ext}.
This implies that for all $(t,x)\in\R_+^*\times\R$,
$Y(\cdot,\circ)G_\nu(t-\cdot,x-\circ)\in\calP_2$. Hence $w(t,x)$
is a
well-defined Walsh integral and the
resulting random field is adapted
to the filtration $\{\calF_s, s\ge0\}$.

Now we shall prove the $L^p(\Omega)$-continuity.
Fix $(t,x)\in\R_+^*\times\R$. Let $B_{t,x}$ and $a$ denote, respectively,
the set
and the constant defined in Proposition~\ref{P2:G-Margin}.
We assume that $ (t',x' )\in B_{t,x}$.
Denote
\[
( t_*,x_* ) = \cases{ \bigl(t',x' \bigr), &\quad $\mbox{if
$t'\le t$,}$
\cr
(t,x), & \quad $\mbox{if $t'> t$,}$}
\quad\mbox{and}\quad
 (\hat{t},\hat{x} ) = \cases{ (t,x), &\quad $\mbox{if $t'\le
t$,}$
\cr
\bigl(t',x' \bigr), & \quad$\mbox{if
$t'> t$.}$}
\]
Set $K_a=[1/a,t+1]\times[-a,a]$. Let $A_a=\sup_{ (s,y )\in K_a}
\Vert Y (s,y )\Vert_p^2$, which is finite by \eqref{E2:Ytxlb}.
By Lemma~\ref{L2:Lp}, we have
\begin{eqnarray*}
&&\bigl\Vert w(t,x)-w \bigl(t',x' \bigr)
\bigr\Vert_p^p
\\
& &\qquad\le 2^{p-1} z_p^p \biggl(\int
_0^{t_*}\int_{\R} \bigl\Vert Y
(s,y )\bigr\Vert_p^2 \bigl(G_\nu(t-s,
x-y)\\
&&\qquad\quad{}-G_\nu\bigl(t'-s , x'-y\bigr)
\bigr)^2 \,\ud s \,\ud y \biggr)^{p/2}
\\
&&\qquad\quad{} +2^{p-1} z_p^p \biggl( \int
_{t_*}^{\hat{t}}\int_{\R}\bigl \Vert Y
(s,y )\bigr\Vert_p^2 G_\nu^2 (
\hat{t}-s,\hat {x}-y )\,\ud s \,\ud y \biggr)^{p/2}
\\
&&\qquad \le2^{p-1} z_p^p \bigl( L_1
\bigl(t,t',x,x' \bigr) \bigr)^{p/2}+2^{p-1}
z_p^p \bigl( L_2 \bigl(t,t',x,x'
\bigr) \bigr)^{p/2}.
\end{eqnarray*}

We first consider $L_1$. Write $L_1=L_{1,1} (t,t',x,x' )
+L_{1,2} (t,t',x,x' )$, where
\begin{eqnarray*}
&&L_{1,1} \bigl(t,t',x,x' \bigr) \\
&&\qquad=
\iint_{ ([0,t_*]\times\R
)\setminus
K_a} \bigl\Vert Y(s,y)\bigr\Vert_p^2
\bigl(G_\nu (t-s,x-y )-G_\nu \bigl(t'-s,x'-y
\bigr) \bigr)^2 \,\ud s\,\ud y,
\\
&&L_{1,2} \bigl(t,t',x,x' \bigr) \\
&&\qquad=
\iint_{ ([0,t_*]\times\R
 )\cap
K_a} \bigl\Vert Y(s,y)\bigr\Vert_p^2
\bigl(G_\nu (t-s,x-y )-G_\nu \bigl(t'-s,x'-y
\bigr) \bigr)^2 \,\ud s\,\ud y.
\end{eqnarray*}
By Proposition~\ref{P2:G-Margin},
%
%e3.8 #&#
\begin{eqnarray}
\label{E2_:LDC-1} &&\sup_{ (t',x' )\in B_{t,x}} \bigl(G_\nu (t-s,x-y
)-G_\nu\bigl(t'-s, x'-y\bigr)
\bigr)^2
\nonumber
\\[-8pt]
\\[-8pt]
\nonumber
&&\qquad\le4 G_\nu^2(t+1-s,x-y),
\end{eqnarray}
for all $s\in[0,t']$ and $|y|\ge a$. Moreover,
\begin{eqnarray*}
&&\iint_{ ([0,t_*]\times\R )\setminus K_a} \bigl\Vert Y (s,y )\bigr\Vert_p^2
G_\nu^2(t+1-s,x-y) \,\ud s\,\ud y \\
&&\qquad\le\bigl\Vert Y(\cdot,
\circ)G_\nu(t+1- \cdot,x-\circ )\bigr\Vert _{M,p}^2<+
\infty.
\end{eqnarray*}
Therefore, Lebesgue's dominated convergence theorem implies that
\[
\lim_{ (t',x' )\rightarrow(t,x)} L_{1,1} \bigl(t,t',x,x'
\bigr)=0.
\]

By Proposition~\ref{P2:G}, for some
constant $C>0$ depending only on $\nu$,
\begin{eqnarray*}
&&L_{1,2} \bigl(t,t',x,x' \bigr)\\
&&\qquad \le
A_a\iint_{ ([0,t_*]\times\R )\cap
K_a} \bigl(G_\nu (t-s,x-y
)-G_\nu\bigl(t'-s, x'-y\bigr)
\bigr)^2\,\ud s\,\ud y
\\
&&\qquad\le A_a C \bigl(\bigl\llvert x-x'\bigr\rrvert +
\sqrt{\bigl|t-t'\bigr|} \bigr).
\end{eqnarray*}
Therefore, $\lim_{ (t',x' )\rightarrow
(t,x)}L_1 (t',t,x,x' ) = 0$.

Now let us consider $L_2$.
Decompose $L_2$ into $L_{2,1} (t,t',x,x' )
+L_{2,2} (t,t',x,x' )$, where
\begin{eqnarray*}
L_{2,1} \bigl(t,t',x,x' \bigr) &=&
\iint_{ ([t_*,\hat{t} ]\times\R )\setminus
K_a} \bigl\Vert Y(s,y)\bigr\Vert_p^2G_\nu
(\hat{t}-s, \hat{x}-y )^2 \,\ud s\,\ud y,
\\
L_{2,2} \bigl(t,t',x,x' \bigr) &=&
\iint_{ ([t_*,\hat{t} ]\times
\R
 )\cap
K_a} \bigl\Vert Y(s,y)\bigr\Vert_p^2
G_\nu (\hat{t}-s, \hat{x}-y )^2 \,\ud s\,\ud y.
\end{eqnarray*}
The proof that $\lim_{ (t',x' )\rightarrow(t,x)}
L_{2,1} (t,t',x,x' ) =0$
is the same as for $L_{1,1}$, except that \eqref{E2_:LDC-1} must be
replaced by
\[
\sup_{ (t',x' )\in B_{t,x}} G_\nu^2 (\hat{t}-s,\hat
{x}-y ) \le G_\nu^2(t+1-s,x-y).
\]
The proof for $L_{2,2}$ is similar to $L_{1,2}$: by Proposition~\ref{P2:G},
\[
L_{2,2} \bigl(t,t',x,x' \bigr)\le
A_a \int_{t_*}^{\hat{t}}\,\ud s \int
_\R G_\nu^2 (\hat{t}-s,\hat{x}-y )
\,\ud y \le A_a C \sqrt{\bigl\llvert t'-t\bigr\rrvert }
\rightarrow0,
\]
as $(t',x')\rightarrow(t,x)$.
Therefore, $\lim_{ (t',x' )\rightarrow
(t,x)}L_2 (t',t,x,x' ) = 0$, which completes the proof.
\end{pf*}

We will need deterministic integral inequalities for the
moments of the solution to \eqref{E2:WalshSI}.
Define $b_p = 1$ if $p=2$ and $b_p=2$ if $p>2$.
Recall the formula $\calL_0$ defined in \eqref{E:Def-Ln}
and define the associated functions $\underline{\calL}_0$ and
$\widetilde{\calL}_{0,p}$ using the
convention~\eqref{E:convention}.

%le3.7 #&#
\begin{lemma}\label{L2:HigherMom}
Suppose that $f(t,x)$ is a deterministic function and $\rho$ satisfies the
growth condition \eqref{E1:LinGrow}.
If the random fields $w$ and $v$ satisfy, for all $t>0$ and $x\in\R$,
\[
w(t,x) = f(t,x) + \iint_{[0,t]\times\R} G_\nu (t-s,x-y ) \rho\bigl(v
(s,y )\bigr)W(\ud s, \ud y),
\]
where we assume that $G_\nu(t-\cdot,x-\circ)\rho(v(\cdot,\circ))
\in
\calP_2$,
%the Walsh integral is well defined,
then for all even
integers $p\ge2$,
\begin{eqnarray*}
\bigl\Vert \bigl(G_\nu\star\rho(v)\dot{W} \bigr) (t,x)
\bigr\Vert_{p}^2& \le &z_p^2\bigl\Vert
G_\nu(t-\cdot,x-\circ) \rho\bigl(v(\cdot,\circ)\bigr)\bigr\Vert
_{M,p}^2
\\
&\le& \frac{1}{b_p} \bigl( \bigl(\Vip^2+\Vert v
\Vert_p^2 \bigr)\star \widetilde{\calL}_{0,p}
\bigr) (t,x). % \label{E2:HigherMom}
\end{eqnarray*}
In particular,
\[
\bigl\Vert w(t,x)\bigr\Vert_p^2 \le b_p
f^2(t,x) + \bigl( \bigl(\Vip^2+\Vert v
\Vert_p^2 \bigr)\star\widetilde{\calL }_{0,p}
\bigr) (t,x),
\]
and, assuming \eqref{E1:lingrow},
%
%e3.9 #&#
\begin{equation}
\label{E:wLbd} \bigl\Vert w(t,x)\bigr\Vert_2^2 \ge
f^2(t,x) + \bigl( \bigl(\underline{\varsigma}^2+\Vert v
\Vert_p^2 \bigr)\star \underline{\calL}_0
\bigr) (t,x).
\end{equation}
\end{lemma}

\begin{pf}
For $p=2$, by the \Ito isometry, \eqref{E1:LinGrow}, and the fact
that $a_{2,\Vip} =1$ and $z_2=1$,
\[
\bigl\Vert w(t,x)\bigr\Vert_2^2 \le f^2(t,x) + \bigl(
\bigl(\Vip^2+\Vert v\Vert_2^2 \bigr)\star
\widetilde{\calL }_{0,2} \bigr) (t,x),
\]
and \eqref{E:wLbd} is obtained similarly.
Now we consider the case $p>2$. Clearly,
\[
\bigl\Vert w(t,x)\bigr\Vert_p^2 \le2\bigl |f(t,x)\bigr|^2 +2
\bigl\Vert \bigl(G_\nu\star \rho(v)\dot{W} \bigr) (t,x)
\bigr\Vert_p^2.
\]
By Lemma~\ref{L2:Lp}, we have that
\[
\bigl\Vert \bigl(G_\nu\star \rho(v)\dot{W} \bigr) (t,x)
\bigr\Vert_p^2 \le z_p^2\bigl\Vert
G_\nu(t-\cdot,x-\circ)\rho\bigl(v(\cdot,\circ)\bigr)
\bigr\Vert_{M,p}^2.
\]
If $\Vip=0$, then $\Vert\rho(v (s,y ))\Vert_p^2
\le\mathit{L}_\rho^2 \Vert v (s,y )\Vert_p^2$.
Otherwise, by \eqref{E1:LinGrow} and subadditivity of the function
$x\mapsto
|x|^{2/p}$,
\[
\bigl\Vert\rho\bigl(v (s,y )\bigr)\bigr\Vert_p^2 \le
\mathit{L}_\rho^2 2^{(p-2)/p} \bigl(\Vip^2
+ \bigl\Vert v (s,y )\bigr\Vert_p^2 \bigr).
\]
Combining these two cases proves that
\begin{eqnarray*}
&&z_p^2 b_p\bigl\Vert G_\nu(t-\cdot,x-
\circ)\rho\bigl(v(\cdot,\circ )\bigr)\bigr\Vert _{M,p}^2
\\
&&\qquad \le z_p^2 \mathit{L}_\rho^2
a_{p,\Vip}^2 \iint_{[0,t]\times\R} G_\nu^2
(t-s,x-y ) \bigl(\Vip^2+\bigl\Vert v (s,y )\bigr\Vert_p^2
\bigr)\,\ud s \,\ud y
\\
& &\qquad= \bigl( \bigl[\Vip^2 +\bigl\Vert v(\cdot,\circ)\bigr\Vert_p^2
\bigr]\star \widetilde{\calL}_{0,p} \bigr) (t,x),
\end{eqnarray*}
because $a_{p,0}^2 = b_p$, and
$a_{p,\Vip}^2= 2^{{(p-2)}/{p}+1} = 2^{2(p-1)/p}$ for $\Vip\ne0$
and $p>2$.
\end{pf}

% \subsection[Main Proofs]{Proof of Existence, Uniqueness and Moment
%Estimates}
%s3.3 #&#
\subsection{Proof of Theorem \texorpdfstring{\protect\ref{T2:ExUni}}{2.4}}
\label{SS2:Existence}

We begin by stating two lemmas.

%le3.8 #&#
\begin{lemma}\label{L2:J0Cont}
The solution $(t,x)\mapsto J_0(t,x)$ to the homogeneous equation~\eqref{E2:Heat-home} with
$\mu\in\calM_H(\R)$ is smooth: $J_0 \in
C^{\infty} (\R_+^*\times\R )$.
If, in addition, $\mu(\ud x) = f(x)\,\ud x$, where $f$ is continuous,
then $J_0
\in C^{\infty} (\R_+^*\times\R ) \cap C (\R
_+\times\R
 )$,
and if $f$ is $\alpha$-H\"older continuous, then
$J_0 \in C^{\infty} (\R_+^*\times\R ) \cap
C_{\alpha/2,\alpha} (\R_+\times\R )$.
\end{lemma}

\begin{pf}
The property $J_0\in C^{\infty} (\R_+^*\times\R )$ is a
slight extension of standard results (see \cite{John91PDE}, (1.14) on
page~210). For more details, we refer the interested reader to
\cite{LeChen13Thesis}, Section~2.6.
We only show here that $J_0 \in
C_{\alpha/2,\alpha} (\R_+\times\R )$ if $\mu(\ud
x)=f(x)\,\ud x$
and $f$
is $\alpha$-H\"older continuous.
Fix $(t,x)$ and $(t',x')\in\R_+\times\R$ with $t'> t$.
By changing variables appropriately, we see that
\[
J_0(t,x)-J_0\bigl(t',x'
\bigr) = \int_\R G_\nu(1,z) \bigl(f (x-\sqrt{t}
z )-f \bigl(x-\sqrt{t'} z \bigr) \bigr)\,\ud z.
\]
By the H\"older continuity of $f$, for some constants $C$ and $C'$,
\[
\bigl\llvert J_0(t,x)-J_0\bigl(t',x
\bigr)\bigr\rrvert \le C \bigl\llvert \sqrt{t}-\sqrt{t'}\bigr\rrvert
^\alpha\int_\R G_\nu(1,z) |z|
^\alpha\,\ud z\le C' \bigl\llvert t'-t\bigr
\rrvert ^{\alpha/2}.
\]
Spatial increments are treated similarly.
\end{pf}

If the initial data is such that $J_0^2(t,x)$ is a constant $v^2$, that is,
$\mu(\ud x)=v\,\ud x$,
then $ (J_0^2\star\calK )(t,x)= (v^2\star\calK
)(t,x) =v^2
\calH(t)$. Clearly,
%
%e3.10 #&#
\begin{equation}\quad
\label{E2:InitK0} \bigl(v^2\star\calL_0 \bigr) (t,x) =
v^2 \lambda^2 \int_0^t
\,\ud s \frac{1}{\sqrt{4\pi\nu s}} \int_\R \,\ud y\, G_{{\nu}/{2}}
(s,y ) = v^2 \lambda^2 \sqrt{\frac{t}{\pi\nu}}.
\end{equation}
For general $J_0^2(t,x)$, we have the following.

%le3.9 #&#
\begin{lemma}\label{L2:InDt}
Fix $\mu\in\calM_H(\R)$.
Suppose $K(t,x)= G_{\nu/2}(t,x) h(t)$ for some nonnegative function
$h(t)$.
Then for all $(t,x)\in\R_+^*\times\R$,
%
%e3.11 #&#
\begin{equation}
\label{E2:InDt-A} \bigl(J_0^2\star K \bigr) (t,x) \le 2
\sqrt{t} \bigl\llvert J_0^*(2t,x)\bigr\rrvert ^2 \int
_0^t \frac{h(t-s)}{\sqrt
{s}}\,\ud s , % \mbox{} ,
\end{equation}
where $J_0^{*}(t,x)=
 (G_{\nu}(t,\cdot) * |\mu| ) (x)$.
In particular, for all $(t,x)\in\R_+^*\times\R$,
%
%e3.12 #&#
%e3.13 #&#
\begin{eqnarray}
\label{E2:InDt} \bigl(J_0^2\star\calK \bigr) (t,x) &\le&
\lambda^2 \sqrt{\pi t/\nu} \bigl\llvert J_0^*(2t,x)\bigr
\rrvert ^2 \biggl(1+ 2 \exp \biggl(\frac{\lambda^4 t}{4\nu} \biggr)
\biggr)<+\infty,
\\
\label{E2:InDt-L0} \bigl(J_0^2\star\calL_0
\bigr) (t,x) &\le& \lambda^2 \sqrt{\pi t/\nu} \bigl\llvert
J_0^*(2t,x)\bigr\rrvert ^2 <+\infty.
\end{eqnarray}
\end{lemma}
\begin{pf}
Assume that $\mu\ge0$. Write $J_0^2 (s,y )$ as a double integral:
%
%e3.14 #&#
\begin{eqnarray}
\label{E2_:upbd} \bigl(J_0^2 \star K \bigr) (t,x) &=& \int
_0^t \,\ud s\int_\R\,\ud y
\iint_{\R^2} G_{\nu}(s,y-z_1)G_{\nu}(s,y-z_2)
\nonumber
\\[-8pt]
\\[-8pt]
\nonumber
&&{}\times G_{{\nu}/{2}} (t-s,x-y ) h(t-s) \mu(\ud z_1)\mu(\ud
z_2).
\nonumber
\end{eqnarray}
Then apply Lemma~\ref{L2:GG} to $G_{\nu}(s,y-z_1)G_{\nu}(s,y-z_2)$ and
integrate over $y$ using the semigroup property of the heat kernel and setting
$\bar{z}=(z_1+z_2)/2$:
%
%e3.15 #&#
\begin{eqnarray}
\label{E2:InDt-Mid} &&\bigl(J_0^2\star K \bigr) (t,x)
\nonumber
\\[-8pt]
\\[-8pt]
\nonumber
&&\qquad =\int
_0^t \,\ud s \iint_{\R^2}
G_{2\nu}(s,z_2-z_1) G_{{\nu}/{2}} (t,x-
\bar{z} )h(t-s) \mu(\ud z_1) \mu (\ud z_2).
\end{eqnarray}
Applying Lemma~\ref{L2:Split} and then integrating over $z_1$ and
$z_2$ proves
\eqref{E2:InDt-A}.
For a signed measure $\mu$, simply replace $\mu$ by $|\mu|$.
The inequality \eqref{E2:InDt-L0} is proved by choosing $h(t) =
\lambda^2(4\pi\nu t)^{-1/2}$.
Finally, \eqref{E2:InDt} follows from \eqref{E2:InDt-A} by taking
$h(t)=\frac{1}{\sqrt{4\pi\nu t}}+
\frac{\lambda^2}{2\nu} e^{{\lambda^4 t}/{(4\nu)}}$ and then
using the change of variable $s=u^2/a$ to see that
%
%e3.16 #&#
\begin{equation}
\label{E2:ExpDivSqrt} \int_0^t \frac{e^{a (t-s)}}{\sqrt{s}}
\,\ud s =\sqrt{\pi/a} e^{a t} \Erf (\sqrt{a t} )\le\sqrt{\pi/a}
e^{a t},\qquad a>0.
\end{equation}
This completes the proof.
\end{pf}

Comparing the proofs of \eqref{E2:InDt} and \eqref{E2:InDt-L0},
we can see that $ (J_0^2 \star\calK ) (t,x)<\infty
$ if and only if $ (J_0^2 \star\calL_0  ) (t,x)<\infty$:
the main issue is the integrability around $t=0$ caused by
the factor $\frac{1}{\sqrt{t}}$ in $\calL_0$.

\begin{pf*}{Proof of Theorem~\ref{T2:ExUni}}
Fix an even integer $p\ge2$.

\textit{Step} 1.
Define $u_0(t,x) = J_0(t,x)$. By Lemma~\ref{L2:J0Cont}, $u_0(t,x)$ is
a well defined and continuous function over $(t,x) \in\R_+^*\times\R$.
We shall now apply Proposition~\ref{P2:Picard} with $Y=\rho(u_0)$.
We check the three properties that it
requires. Properties (i) and (ii) are trivially satisfied since $Y$ is
deterministic and continuous over $\R_+^*\times\R$. Property~(iii)
is also
true since, by
Lemma~\ref{L2:HigherMom},
%
%e3.17 #&#
\begin{equation}\quad
\label{E2_:step1} b_p z_p^2 \bigl\Vert\rho
\bigl(u_0(\cdot,\circ) \bigr) G_\nu(t-\cdot ,x-\circ)
\bigr\Vert_{M,p}^2 \le \bigl( \bigl[\Vip^2+J_0^2
\bigr] \star \widetilde{\calL}_{0,p} \bigr) (t,x) ,
\end{equation}
which is finite by \eqref{E2:InitK0} and Lemma~\ref{L2:InDt}.
Hence, the following Walsh integral is well defined and
is an adapted random field
\[
I_1(t,x) = \iint_{[0,t]\times\R} \rho \bigl(u_0 (s,y )
\bigr) G_\nu (t-s,x-y ) W (\ud s,\ud y ).
\]
The continuity of the deterministic function $ (s,y )\mapsto
\rho(u_0 (s,y ))$ implies its local $L^p(\Omega)$-boundedness
[in the
sense of \eqref{E2:Ytxlb}]. So $(t,x)\mapsto I_1(t,x)$ is
$L^p(\Omega)$-continuous on $\R_+^*\times\R$ by Proposition~\ref
{P2:Picard}.

Define $u_1(t,x)=
J_0(t,x)+I_1(t,x)$.
Since $J_0(t,x)$ is continuous on $\R_+^*\times\R$,
$u_1(t,x)$ is $L^p(\Omega)$-continuous on $\R_+^*\times\R$.
Now we estimate its moments. By \Itos isometry,
\[
\bigl\Vert I_1(t,x)\bigr\Vert_2^2= \bigl\Vert\rho
\bigl(u_0(\cdot,\circ) \bigr) G_\nu(t-\cdot,x-\circ )
\bigr\Vert_{M,2}^2,
\]
which equals $ ( [\vv^2+J_0^2 ] \star
\calL_0 )(t,x)$ for the quasi-linear case \eqref{E1:qlinear},
and is
bounded from above [see \eqref{E2_:step1} with $b_2 z_2^2 =1$] and
below [if
$\rho$ additionally satisfies \eqref{E1:lingrow}], in which case
\[
\bigl( \bigl[\Vip^2+J_0^2 \bigr] \star
\underline{\calL}_0 \bigr) (t,x) \le \bigl\Vert I_1(t,x)
\bigr\Vert_2^2 \le \bigl( \bigl[\Vip^2+J_0^2
\bigr] \star\overline{\calL}_0 \bigr) (t,x).
\]
Since $J_0(t,x)$ is deterministic and since $\E [I_1(t,x) ]=0$,
$\Vert u_1(t,x)\Vert_2^2 = J_0^2(t,x)
+\Vert I_1(t,x)\Vert_2^2$, and by Lemma~\ref{L2:HigherMom},
\begin{eqnarray*}
\bigl\Vert u_1(t,x)\bigr\Vert_p^2 & \le&
b_p J_0^2(t,x)+ \bigl( \bigl(
\Vip^2+ J_0^2 \bigr)\star\widetilde{\calL
}_{0,p} \bigr) (t,x)
\\
&\le & b_p J_0^2(t,x)+ \bigl( \bigl(
\Vip^2+b_p J_0^2 \bigr)\star
\widetilde{\calK}_p \bigr) (t,x),
\end{eqnarray*}
since $b_p\ge1$ and $\widetilde{\calL}_{0,p}\le\widetilde{\calK
}_p$ by
\eqref{E:Def-K}.

In summary, $u_1$ is a well-defined random field that satisfies (with $k=1$)
the four properties (1)--(4) described just below in step 2.\vadjust{\goodbreak}

\textit{Step} 2.
Assume by induction that for all $k\le n$ and $(t,x)\in\R_+^*\times
\R$, the
Walsh integral
\[
I_k(t,x) = \iint_{[0,t]\times\R} \rho \bigl(u_{k-1} (s,y
) \bigr) G_\nu (t-s,x-y ) W (\ud s,\ud y )
\]
is well defined such that:
\begin{longlist}[(1)]
\item[(1)]$u_k:=J_0+I_k$ is adapted to the filtration $\{\calF_t\}_{t>0}$.
\item[(2)] The function $(t,x)\mapsto u_k(t,x)$ from $\R_+^*\times\R$ into
$L^p(\Omega)$ is continuous.
\item[(3)]$\E [u_k^2(t,x) ] = J_0^2(t,x)+\sum_{i=0}^{k-1}
 ( [\vv^2+J_0^2 ] \star\calL_i  )(t,x)$ for the
quasi-linear case and it is bounded from above and below [if
$\rho$ satisfies \eqref{E1:lingrow}] by
\begin{eqnarray*}
J_0^2(t,x)+\sum_{i=0}^{k-1}
\bigl( \bigl[\underline{\varsigma }^2+J_0^2
\bigr]\star \underline{\calL}_i \bigr) (t,x) &\le& \E
\bigl[u_k^2(t,x) \bigr]\\
 &\le& J_0^2(t,x)+
\sum_{i=0}^{k-1} \bigl( \bigl[
\Vip^2+J_0^2 \bigr] \star \overline{
\calL}_i \bigr) (t,x).
\end{eqnarray*}
\item[(4)]$\Vert u_k(t,x)\Vert_p^2 \le b_p J_0^2(t,x)+
 ( (\Vip^2+b_p J_0^2 )\star
\widetilde{\calK}_p )(t , x)$.
\end{longlist}

We are now going to define $u_{n+1}(t,x)$. We shall apply Proposition~\ref{P2:Picard} again, with
$Y(s,y)=\rho (u_n (s,y ) )$, by
verifying
the three properties that it requires. Properties (i) and (ii) are clearly
satisfied by the induction assumptions (1)
and (2). By Lemma~\ref{L2:HigherMom} and the induction assumptions, we establish
property (iii):
%
%e3.18 #&#
\begin{eqnarray}
\label{E2_:step2Mp}
&&b_p z_p^2 \bigl\Vert\rho
\bigl(u_n(\cdot,\circ) \bigr) G_\nu(t-\cdot ,x-\circ)
\bigr\Vert_{M,p}^2\nonumber\\
&&\qquad\le \bigl( \bigl(\Vip^2+\Vert
u_n\Vert_p^2 \bigr)\star \widetilde{
\calL}_{0,p} \bigr) (t,x)
\nonumber
\\[-8pt]
\\[-8pt]
\nonumber
&&\qquad\le \bigl( \bigl[\Vip^2+b_p
J_0^2 + \bigl(\Vip^2+b_p
J_0^2 \bigr)\star \widetilde{\calK}_p
\bigr]\star\widetilde{\calL}_{0,p} \bigr) (t,x)
\\
&&\qquad= \bigl( \bigl[\Vip^2+b_p J_0^2
\bigr]\star\widetilde{\calK}_p \bigr) (t,x),\nonumber
\end{eqnarray}
by \eqref{E:KS-Kn-K0}, and this is finite by Lemma~\ref{L2:InDt}.

Hence, for all $(t,x)\in\R_+^*\times\R$,
$\rho (u_n(\cdot,\circ)  ) G_\nu(t-\cdot,x-\circ) \in
\calP_{p}$
and the Walsh integral
\[
I_{n+1}(t,x) = \iint_{[0,t]\times\R} \rho \bigl(u_n (s,y )
\bigr) G_\nu (t-s,x-y ) W (\ud s,\ud y )
\]
is a well defined and adapted random field. By assumption (2),
$ (s,y )\mapsto
\rho(u_n (s,y ))$ is $L^p(\Omega)$-continuous, so Proposition~\ref{P2:Picard} implies that $(t,x)\mapsto I_{n+1}(t,x)$ is also
$L^p(\Omega)$-continuous. Define
\[
u_{n+1}(t,x)= J_0(t,x) + I_{n+1}(t,x).\vadjust{\goodbreak}
\]

Now we estimate the moments of $u_{n+1}(t,x)$. By Lemma~\ref
{L2:HigherMom} and
\eqref{E2_:step2Mp},
\[
\bigl\Vert u_{n+1}(t,x)\bigr\Vert_p^2 \le b_p
J_0^2(t,x)+ \bigl( \bigl(\Vip^2+b_p
J_0^2 \bigr)\star \widetilde{\calK}_p
\bigr) (t,x).
\]
As for the second moment, by Lemma~\ref{L2:HigherMom},
\begin{eqnarray*}
J_0^2(t,x)+ \bigl( \bigl[\underline{
\varsigma}^2+\Vert u_n\Vert _2^2
\bigr] \star \underline{\calL}_0 \bigr) (t,x) &\le& \E
\bigl[u_{n+1}^2(t,x)\bigr] \\
&\le& J_0^2(t,x)+
\bigl( \bigl[\Vip^2+\Vert u_n\Vert_2^2
\bigr] \star \overline{\calL}_0 \bigr) (t,x).
\end{eqnarray*}
Substituting the bounds from induction assumption (3) gives
\begin{eqnarray*}
J_0^2(t,x)+\sum_{i=0}^n
\bigl( \bigl[\underline{\varsigma }^2+J_0^2
\bigr] \star \underline{\calL}_i \bigr) (t,x) &\le& \E
\bigl[u_{n+1}^2(t,x)\bigr]\\
& \le& J_0^2(t,x)+
\sum_{i=0}^n \bigl( \bigl[
\Vip^2+J_0^2 \bigr] \star \overline{
\calL}_i \bigr) (t,x).
\end{eqnarray*}
In the quasi-linear case, the inequalities become the equality
\[
\E\bigl[u_{n+1}^2(t,x)\bigr] = J_0^2(t,x)+
\sum_{i=0}^n \bigl( \bigl[
\Vip^2+J_0^2 \bigr] \star
\calL_i \bigr) (t,x).
\]
Therefore, the four properties (1)--(4) also hold for $k=n+1$.

\textit{Step} 3. We claim that
for all $(t,x)\in\R_+^*\times\R$, the sequence
$ \{u_n(t,x) \}_{n\in\bbN}$ is a
Cauchy sequence in $L^p(\Omega)$, and we will use $u(t,x)$ to denote
its limit.
To prove this claim, define $F_n(t,x)=\Vert u_{n+1}(t,x)-u_n (t,x)\Vert_p^2$.
For $n\ge1$, by Lemma~\ref{L2:Lp} and the Lipschitz continuity of
$\rho$,
\begin{eqnarray}
F_n(t,x) \le (F_{n-1}\star\widecheck{\calL}_{0,p}
) (t,x)\nonumber \\
\eqntext{\mbox{with $ \widecheck{\calL}_{0,p}(t,x) :=
\calL_0 \bigl(t,x;\nu, z_p \max (\LIP_\rho,
a_{p,\Vip}\mathit{L}_\rho ) \bigr)$.}}
\end{eqnarray}
By analogy with the convention \eqref{E:convention},
the functions $\widecheck{\calL}_{n,p}(t,x)$ and
$\widecheck{\calK}(t,x)$ are defined by the same parameters as
$\widecheck{\calL}_{0,p}(t,x)$.
For the case $n=0$, we need to use the linear growth
condition \eqref{E1:LinGrow} instead: By Lemma~\ref{L2:HigherMom},
\[
F_0(t,x) \le \bigl( \bigl[ \Vip^2+J_0^2
\bigr] \star\widetilde{\calL}_{0,p} \bigr) (t,x) \le \bigl( \bigl[
\Vip^2+J_0^2 \bigr] \star\widecheck{
\calL}_{0,p} \bigr) (t,x).
\]
Then apply the above relation recursively:
\begin{eqnarray*}
F_n(t,x) &\le& (F_{n-1}\star\widecheck{\calL}_{0,p}
) (t,x) \le\cdots\le \bigl( \bigl[ \Vip^2+J_0^2
\bigr] \star\widecheck{\calL }_{n,p} \bigr) (t,x)\\
&\le& \bigl( \bigl[
\Vip^2+J_0^2 \bigr] \star \widecheck{
\calL}_{0,p} \bigr) (t,x) B_n(t),
\end{eqnarray*}
by \eqref{E:KS-LB}.
Now by Proposition~\ref{P2:K}, for all $(t,x)\in\R_+^*\times\R$
fixed and
all
$m\in\bbN^*$,
\[
\sum_{i=0}^\infty\bigl\llvert
F_i(t,x)\bigr\rrvert ^{1/m} \le \bigl\llvert \bigl( \bigl[
\Vip^2+J_0^2 \bigr] \star\widecheck{
\calL}_{0,p} \bigr) (t,x)\bigr\rrvert ^{1/m} \sum
_{i=0}^\infty \bigl\llvert B_i(t)\bigr
\rrvert ^{1/m}<+\infty,
\]
which proves that $\{u_n(t,x)\}_{n\in\bbN}$ is a Cauchy sequence in
$L^p(\Omega)$ by taking $m=2$.

The moments estimates \eqref{E2:SecMom-Up}, \eqref{E2:SecMom-Lower} and
\eqref{E2:SecMom} can be obtained simply by letting $n\rightarrow
+\infty
$ in the
conclusions (3) and (4) of the previous step and using~\eqref{E:Def-K} and
\eqref{E2:H}.
Now let us prove the $L^p(\Omega)$-continuity.
For all $a>0$, set $K_a:=[1/a,a]\times[-a,a]$.
Since $B_n(t)$ is nondecreasing, the above $L^p(\Omega)$ limit is
uniform over
$K_a$ because
\[
\sum_{i=0}^\infty\sup_{(t,x)\in K_a}
\bigl\llvert F_i(t,x)\bigr\rrvert ^{1/m} \le \Biggl(\sum
_{i=0}^\infty \bigl\llvert
B_i(a)\bigr\rrvert ^{1/m} \Biggr) \sup
_{(t,x)\in K_a} \bigl\llvert \bigl( \bigl[ \Vip^2+J_0^2
\bigr] \star\widecheck{\calL}_{0,p} \bigr) (t,x)\bigr\rrvert
^{1/m}.
\]
By \eqref{E2:InitK0}, \eqref{E2:InDt-L0} and the continuity of
$(t,x)\mapsto J_0^*(2t,x)$ over $\R_+^*\times\R$  (see Lemma~\ref
{L2:J0Cont}),
we see that the right-hand side is finite.
Hence,\break  $\sum_{i=0}^\infty\sup_{(t,x)\in K_a}
\llvert F_i(t,x)\rrvert ^{1/m}<+\infty$, which implies that the function
$(t,x)\mapsto u(t,x)$ from $\R_+^*\times\R$ into $L^p(\Omega)$ is continuous
over $K_a$ since each $u_n(t,x)$ is so.
As $a$ can be arbitrarily large, we have then proved the
$L^p(\Omega)$-continuity of $(t,x)\mapsto u(t,x)$ over $\R_+^*\times
\R$.

The following inequality, which will be used in step 4, is a direct consequence
of the upper bound (4) of step 2 and \eqref{E:KS-Kn-K0}:
%
%e3.19 #&#
\begin{equation}
\label{E2_:u-K0} \bigl( \bigl[\Vip^2+\Vert u\Vert_p^2
\bigr]\star\widetilde{\calL }_{0,p} \bigr) (t,x) \le \bigl( \bigl[
\Vip^2+b_p J_0^2 \bigr]\star
\widetilde{\calK }_p \bigr) (t,x).
\end{equation}

\textit{Step \textup{4} \textup{(}Verifications\textup{)}}.
Now we shall verify that
$\{u(t,x), (t,x)\in\R_+^*\times\R\}$
defined in the previous step is indeed a solution to the stochastic integral
equation~\eqref{E2:WalshSI} in the sense of Definition~\ref{D2:Solution}.
Clearly, $u$ is adapted and jointly-measurable, and hence it satisfies
(1) and
(2) of Definition~\ref{D2:Solution}.
The continuity of the function $(t,x)\mapsto u(t,x)$ from $\R
_+^*\times
\R$ into
$L^2 (\R )$ proved
in step 3, Proposition~\ref{P2:Picard} applied to $Y=\rho(u_n)$ and
\eqref{E2_:u-K0} imply (3) of
Definition~\ref{D2:Solution}.
So we only need to verify that $u$ satisfies
(4) of Definition~\ref{D2:Solution}, that is, $u(t,x)$ satisfies
\eqref{E2:WalshSI} a.s., for all $(t,x)\in\R_+^*\times\R$.

We shall apply Proposition~\ref{P2:Picard} with
$Y(s,y)=\rho(u (s,y ))$ by
verifying the three properties that it requires.
Properties (i) and (ii) are satisfied by (1) and (2) in the conclusion part
of step 3. Property (iii) is also true since, by Lemma~\ref{L2:HigherMom}
and also~\eqref{E2_:u-K0},
\begin{eqnarray*}
b_p z_p^2 \bigl\Vert\rho \bigl(u (\cdot,\circ)
\bigr) G_\nu(t-\cdot ,x-\circ)\bigr\Vert_{M,p}^2 &\le&
\bigl( \bigl(\Vip^2+\Vert u\Vert_p^2 \bigr)
\star\widetilde{\calL }_{0,p} \bigr) (t, x)\\
&\le& \bigl( \bigl[
\Vip^2+b_p J_0^2 \bigr]\star
\widetilde{\calK}_p \bigr) (t,x) ,
\end{eqnarray*}
which is finite by Lemma~\ref{L2:InDt}. Hence,
\[
\rho \bigl(u(\cdot,\circ) \bigr) G_\nu(t-\cdot,x-\circ) \in
\calP_{p}  \qquad\mbox{for all $(t,x)\in\R_+^*\times\R$} ,
\]
and the following Walsh integral is well defined and is an adapted
random field
\[
% \label{E2_:SI-u}
I(t,x):=\iint_{[0,t]\times\R} \rho \bigl(u (s,y ) \bigr)
G_\nu (t-s,x-y ) W(\ud s,\ud y).
\]
Furthermore, by the last part of Proposition~\ref{P2:Picard},
$(t,x)\mapsto I(t,x)$ is $L^p(\Omega)$-continuous, since by conclusion
(2) of
step 3, $(t,x)\mapsto u(t,x)$ is $L^p(\Omega)$-\break continuous.

By step 3,
\[
u_n(t,x) = J_0(t,x) + \iint_{[0,t]\times\R}
G_\nu (t-s,x-y ) \rho \bigl(u_{n-1} (s,y ) \bigr) W (\ud s,
\ud y )
\]
with $u_n(t,x)$ converging to $u(t,x)$ in $L^p(\Omega)$. We only need
to show
that
the right-hand side converges in $L^p(\Omega)$ to $J_0(t,x) +I(t,x)$.
In fact, by Lemma~\ref{L2:Lp},
\begin{eqnarray*}
% \label{E2_:Un-un}
&&\biggl\Vert\iint_{[0,t]\times\R} G_\nu (t-s,x-y ) \bigl[\rho
\bigl(u (s,y ) \bigr)-\rho \bigl(u_n (s, y ) \bigr) \bigr] W (\ud s,
\ud y )\biggr\Vert_p^2
\\
&&\qquad\le z_p^2 \LIP_\rho^2
\iint_{[0,t]\times\R} G_\nu^2 (t-s,x-y ) \bigl\Vert u (s,y
)-u_n (s,y )\bigr\Vert_p^2 \,\ud s\,\ud y.
\end{eqnarray*}
Now apply Lebesgue's dominated convergence theorem to conclude that
the above integral tends to zero as $n\rightarrow\infty$ because
(i) for all $ (s,y )\in\,]0,t]\times\R$,
$\Vert u_n (s,y )-u (s,y )\Vert_p^2\rightarrow
0$ as
$n\rightarrow
+\infty$;
(ii) by step 2,
\[
\bigl\Vert u_n(s, y)\bigr\Vert_p^2 \le b_p
J_0^2(s, y)+ \bigl( \bigl[\Vip^2+b_p
J_0^2 \bigr]\star \widetilde{\calK}_p
\bigr) (s, y),
\]
and by step 3, the same upper bound applies to $\Vert u(s,y)\Vert_p^2$. Finally,
by Lemma~\ref{L2:InDt} and \eqref{E:KS-Kn-K0}, the above upper bound,
multiplied
by
$G_\nu^2(t-s,x-y)$, is integrable over $[0,t]\times\R$.
This finishes the proof of the existence
part of Theorem~\ref{T2:ExUni} with the moment estimates.

% \paragraph{Step 5 (Uniqueness)}
\textit{Step \textup{5} \textup{(}Uniqueness\textup{)}.}
Let $u$ and $v$ be two solutions to \eqref{E2:WalshSI} (in the sense of
Definition~\ref{D2:Solution}) with the same
initial data, and denote $w(t,x):= u(t,x)-v(t,x)$.
The $L^2(\Omega)$-continuity---property (3) of Definition~\ref{D2:Solution}---guarantees that both $(t,x)\mapsto u(t,x)$ and
$(t,x)\mapsto v(t,x)$ are
$L^2(\Omega)$-continuous since $(t,x)\mapsto J_0(t,x)$ is continuous
by Lemma~\ref{L2:J0Cont}.
Then $w(t,x)$ is well defined and the
function $(t,x)\mapsto w(t,x)$ is $L^2(\Omega)$-continuous.
Writing $w(t,x)$ explicitly
and then taking the second moment, by \Itos isometry and the Lipschitz
condition on $\rho$, we have
%
%e3.20 #&#
\begin{eqnarray}
\label{E2:Uniqueness} \E\bigl[w(t,x)^2\bigr] \le \bigl(\E
\bigl[w^2\bigr]\star \calL_0^* \bigr) (t,x)
\nonumber
\\[-8pt]
\\[-8pt]
\eqntext{\mbox{where
$\calL_0^*(t,x) := \calL_0 (t,x;\nu,
\LIP_\rho )$.}}
\end{eqnarray}
Now we convolve both sides with respect to
$\calK^*(t,x):=\calK(t,x;\nu,\LIP_\rho)$ and use~\eqref{E:KS-Kn-K0} to obtain
\begin{eqnarray*}
\bigl(\E\bigl[w^2\bigr]\star\calK^*\bigr) (t,x) &\le& \bigl(\E
\bigl[w^2\bigr]\star\calL_0^*\star\calK^*\bigr) (t,x)\\
& =&
\bigl(\E\bigl[w^2\bigr]\star\calK^*\bigr) (t,x) - \bigl(\E
\bigl[w^2\bigr]\star \calL_0^*\bigr) (t,x).\vadjust{\goodbreak}
\end{eqnarray*}
So $(\E[w^2]\star\calL_0^*)(t,x) \equiv0$,
which implies by \eqref{E2:Uniqueness} that $\E[w(t,x)^2]=0$ for all
$(t,x)\in\R_+^*\times\R$. Therefore, we conclude
that for all $(t,x)\in\R_+^*\times\R$,
$u(t,x)=v(t,x)$ a.s.

\textit{Step \textup{6} \textup{(}Two-point correlations\textup{)}}.
In this last step, we prove the properties \eqref{E2:TP-Up},
\eqref{E2:TP-Lower} and \eqref{E2:TP} of the two-point correlation
function.
Let $u(t,x)$ be the
solution to \eqref{E2:WalshSI}. Fix $\tau\ge t\in\R_+^*$ and
$x,y\in\R$.
Consider the
$L^2(\Omega)$-martingale $\{U(s;t,x), s\in[0,t]\}$ defined by
\[
U(s;t,x):= J_0(t,x) + \int_0^s
\int_\R G_\nu(t-r,x-z)\rho\bigl(u(r,z)\bigr)W
(\ud r,\ud z ).
\]
Then $U(t;t,x)=u(t,x)$ and $\E [U(s;t,x) ]=J_0(t,x)$.
Similarly, we
define the
martingale
$\{U(s;\tau,y), s\in[0,\tau]\}$. The mutual variation process of
these two
martingales is, for all $s\in[0,t]$,
\[
\bigl\langle U(\cdot;t,x),U(\cdot;\tau,y) \bigr\rangle_s= \int
_0^s \,\ud r\int_\R\,\ud z
\rho^2 \bigl(u(r,z) \bigr) G_\nu(t-r,x-z)G_\nu
(\tau -r,y-z ).
\]
Hence, by \Itos lemma, for every $s\in[0,t]$,
$\E [U(s;t,x)U(s;\tau,y) ]$ is equal to
\[
J_0(t,x)J_0 (\tau,y ) + \int_0^s
\,\ud r\int_\R\,\ud z \E \bigl[\rho^2 \bigl(u(r,z)
\bigr) \bigr]G_\nu(t-r,x-z)G_\nu (\tau-r, y-z ).
\]
Finally, we choose $s= t$ and note that
$\E [u(t,x)u(\tau,y) ]=\E [u(t,x)U(t;\tau,y)
]$ to get
%
%e3.21 #&#
\begin{eqnarray}
\label{E2_:TP} %
&& \E \bigl[u(t,x)u (\tau,y ) \bigr] \nonumber\\
&&\qquad=
J_0(t,x)J_0 (\tau ,y )
\\
&&\qquad\quad{}+ \int_0^t \,\ud r \int_\R
\,\ud z \bigl\Vert\rho\bigl(u(r,z)\bigr)\bigr\Vert_2^2
G_\nu(t-r,x-z)G_\nu (\tau -r,y-z ). \nonumber
\end{eqnarray}
Then \eqref{E2:TP-Up}, \eqref{E2:TP-Lower} and \eqref{E2:TP} follow from
Lemma~\ref{L:IntIntGG}.
This completes the proof of Theorem~\ref{T2:ExUni}.
\end{pf*}

%s3.4 #&#
\subsection{Proofs of Corollary \texorpdfstring{\protect\ref{C2:TP-Delta}}{2.8} and
Proposition \texorpdfstring{\protect\ref{P2:D-Delta}}{2.11}}\label{Ss:EUM-OtherProof}
\mbox{}
\begin{pf*}{Proof of Corollary~\ref{C2:TP-Delta}}
In this case, $J_0(t,x)=G_\nu(t,x)$ and\break  $\lambda^2 J_0^2(t, x) =
\calL_0(t,x)$.
So, by \eqref{E2:SecMom} and \eqref{E:KS-Kn-K0},
\[
\E \bigl[\bigl|u(t,x)\bigr|^2 \bigr] = \frac{1}{\lambda^2}
\calL_0(t,x)+ \frac{1}{\lambda^2} (\calL_0\star\calK )
(t,x) + \vv^2\calH(t),
\]
yielding \eqref{E2:SecMom-Delta}.
By \eqref{E2:TP} [see also the equivalent formula
\eqref{E2_:TP}], $\E [u(t,x)\times u(t,y) ]=J_0(t,x)J_0
(t,y )
+ \lambda^2 I$, where
\[
I=\int_0^t \,\ud r \int_\R
\,\ud z \biggl(\vv^2 +\frac{1}{\lambda^2}\calK(r,z)+\vv^2
\calH(r) \biggr) G_\nu(t-r,x-z)G_\nu(t-r,y-z).
\]
Use Lemma~\ref{L2:GG} to replace the last two factors by
$G_{\nu/2}(t-r,z-(x+y)/2) G_{2\nu}(t-r,x-y)$, so that $z$ appears in
only one
factor. Then use formula \eqref{E:K2} and the semigroup property of
the heat
kernel
to see that
\begin{eqnarray*}
&&\frac{1}{\lambda^2} \bigl( \calK(r,\cdot) * G_{{\nu}/{2}}(t-r,\cdot) \bigr)
\biggl(\frac{x+y}{2} \biggr)\\
&&\qquad = G_{{\nu}/{2}} \biggl(t,\frac{x+y}{2}
\biggr) \biggl(\frac{1}{\sqrt{4\pi\nu
r}}+\frac{\lambda^2}{4\nu} \bigl(1+\calH(r) \bigr)
\biggr).
\end{eqnarray*}
Therefore,
\begin{eqnarray*}
I&=&\int_0^t G_{2\nu} (t-r,x-y )
\biggl( \biggl(\vv^2 +\frac{\lambda^2}{4\nu} G_{{\nu}/{2}} \biggl(t,
\frac{x+y}{2} \biggr) \biggr) \bigl(\calH(r)+1 \bigr)\\
&&{} + G_{{\nu}/{2}}
\biggl(t,\frac{x+y}{2} \biggr)\frac{1}{\sqrt{4\pi
\nu
r}} \biggr) \,\ud r.
\end{eqnarray*}
Then apply Lemmas \ref{L2:IntHG} and \ref{L2:IntGG} to evaluate the remaining
integrals over $\ud r$.
\end{pf*}

\begin{pf*}{Proof of Proposition~\ref{P2:D-Delta}}
If $\mu=\delta_0'$, then $J_0(t,x)=\frac{\partial}{\partial x}
G_\nu
(t,x) = -
\frac{x}{\nu t}
G_\nu(t,x)$. Suppose that \eqref{E2:WalshSI} has a random field
solution $u(t,x)$. Fix $(t,x)\in\R_+^*\times\R$.
Hence, by \eqref{E2:WalshSI} and \Itos isometry [see \eqref{E2:fint}],
$\Vert u(t,x)\Vert_2^2 \ge J_0^2(t,x)$.
Therefore,
\[
\bigl(G_\nu^2 \star\bigl\Vert\rho(u)\bigr\Vert_2^2
\bigr) (t,x) =\lambda^2 \bigl(G_\nu^2 \star
\Vert u\Vert_2^2 \bigr) (t,x)\ge \lambda^2
\bigl(G_\nu^2 \star J_0^2 \bigr)
(t,x).
\]
Write out the space--time convolution and apply the formulas in Lemma~\ref{L2:GG}
to see that it equals
\begin{eqnarray*}
&&\frac{G_{{\nu}/{2}}(t,x)}{4\pi\nu^3} \int_0^t\,\ud s
\frac{1}{s^2 \sqrt{s(t-s)}} \int_\R\,\ud y\, y^2
G_{{\nu}/{2}} \biggl(\frac{s(t-s)}{t},y-\frac{s}{t}x \biggr)
\\
&&\qquad= \frac{G_{{\nu}/{2}}(t,x)}{4\pi\nu^3} \int_0^t
\frac{1}{s^2 \sqrt{s(t-s)}}\E \biggl[Z^2 + \frac{s^2 x^2}{t^2} \biggr]\,\ud s,
\end{eqnarray*}
where $Z\sim N (0,\nu s(t-s)/(2t) )$ is a Normal random variable.
The expectation is equal to $\frac{\nu s}{2} -\frac{\nu
s^2}{2t}+\frac{s^2 x^2}{t^2}$, and the last two terms yield a finite integral,
but not the first term, so we conclude that
$ (G_\nu^2 \star\Vert\rho(u)\Vert_2^2 )(t,x)\ge+\infty$.
This violates property~(3) of Definition~\ref{D2:Solution}.
\end{pf*}

%s4 #&#
\section{Upper and lower bounds on growth indices}
\label{S:Proof-Exp}

Because the quasi-linear case corresponds to the case where
$\mathit{L}_\rho=\mathit{l}_\rho=|\lambda|$ and
$\Vip=\underline{\varsigma}=\vv$, part~(3) of Theorem~\ref
{T2:Growth} is a direct consequence
of parts (1) and (2). Hence, in the following, we only need to prove
parts (1)
and (2). We first recall a lemma.

%le4.1 #&#
\begin{lemma}[(\cite{ConusKhosh10Farthest})]\label{L2:Growth}
For $2\le a\le b<+\infty$, $\overline{\lambda}(a)\le\overline
{\lambda
}(b)$ and
$\underline{\lambda}(a)\le\underline{\lambda}(b)$.
\end{lemma}

%s4.1 #&#
\subsection{Proof of the lower bound} \label{SS2:ExpInd-Low}
By the moment formula \eqref{E2:SecMom-Lower}, we can bound the second
moment of $u(t,x)$ from below provided we have a lower bound on $J_0(t,x)$.
The next lemma gives such a bound.

%le4.2 #&#
\begin{lemma}\label{L2:pLbd}
Assume that $\mu\in\calM_{H,+} (\R )$ and $\mu\ne0$.
For any $\varepsilon>0$ and $\xi\in\,]0,\nu[$, there exists a constant
$a_{\varepsilon,\xi,\nu}>0$ such that
\[
J_0(t,x)\ge a_{\varepsilon,\xi,\nu} 1_{[\varepsilon,+\infty[}(t)
G_\xi(t,x)\qquad\mbox{for all $t\ge \varepsilon $ and $x\in\R$.}
\]
\end{lemma}

\begin{pf}
It suffices to prove that
\[
g(t,x):= \frac{J_0(t,x)}{G_\xi(t,x)} = \sqrt{\xi/\nu} \int_\R
\exp \biggl(-\frac{(x-y)^2}{2\nu t}+\frac{x^2}{2\xi t} \biggr) \mu (\ud y)
\]
is strictly bounded away from zero for $t\in
[\varepsilon,+\infty[$ and $x\in\R$. Notice that for $0<\xi< \nu$,
\[
-\frac{(x-y)^2}{2\nu t}+\frac{x^2}{2\xi t} = -\frac{(\xi-\nu) [x-{\xi
y}/{(\xi-\nu)} ]^2}{2\nu\xi t} +
\frac{y^2}{2(\xi-\nu)t} \ge -\frac{y^2}{2(\nu-\xi)t}.
\]
Thus, for $t\in[\varepsilon,+\infty[$,
\begin{eqnarray*}
g(t,x) &\ge&\sqrt{\xi/\nu} \int_\R e^{-{y^2}/{(2(\nu-\xi) t)}} \mu (
\,\ud y ) \ge\sqrt{\xi/\nu} \int_\R e^{-{y^2}/{(2(\nu-\xi)
\varepsilon)}} \mu (
\ud y )
\\
&= &\sqrt{2\pi(\nu-\xi)\xi \varepsilon/\nu} \bigl(G_{\nu-\xi}(\varepsilon,
\cdot)*\mu \bigr) (0) =: a_{\varepsilon,\xi,\nu} ,
\end{eqnarray*}
which proves the lemma. We remark that
$ (G_{\nu-\xi}(\varepsilon,\cdot)*\mu )(0)$ is
strictly positive and finite because
$\mu\in\calM_{H,+} (\R )$,
$\mu\ne0$, and $G_{\nu-\xi}(\varepsilon,y)>0$ for all $y\in\R$.
\end{pf}

\begin{pf*}{Proof of Theorem~\ref{T2:Growth}(1)}
Due to Lemma~\ref{L2:Growth}, we only need to estimate
$\underline{\lambda}(2)$. Assume first that $\underline{\varsigma
}=0$. Fix $\varepsilon>0$. For
$\xi\in\,]0,\nu[$, use Lemma~\ref{L2:pLbd} to choose
$a=a_{\varepsilon,\xi
,\nu}>0$
such that
\[
J_0(t,x) \ge I_{0,l}(t,x) := a 1_{[\varepsilon,+\infty[}(t)
G_{\xi}(t,x).
\]
By \eqref{E:K} and since $\Phi(0)=1/2$,
\[
\underline{\calK}(t,x)\ge \frac{\mathit{l}_\rho^4}{4\nu} K(t,x)\qquad\mbox{with $K(t,x):=
G_{{\nu}/{2}}(t,x) e^{{\mathit{l}_\rho^4
t}/{(4\nu)}}$}.
\]
Set $f(t,x)= \E (u(t,x)^2 )$.
By \eqref{E2:SecMom-Lower} and
the above two inequalities, $f(t,x) \ge
\frac{\mathit{l}_\rho^4}{4\nu} ( I_{0,l}^2\star K  ) (t,x)$.
By Lemma~\ref{L2:GG},
\[
\bigl( I_{0,l}^2\star K \bigr) (t,x) = \frac{a^2}{2\sqrt{\pi\xi}}
e^{{\mathit{l}_\rho^4
t}/{(4\nu)}} \int_\varepsilon^t G_{{\nu}/{2}}
\biggl(t-\frac{(\nu-\xi
)s}{\nu},x \biggr)
 \frac{e^{-{\mathit{l}_\rho^4 s}/{(4\nu)}}}{\sqrt{s}}\,\ud s.
\]
Notice that for $s\in[\varepsilon, t]$,
\[
G_{{\nu}/{2}} \biggl(t-\frac{(\nu-\xi)s}{\nu},x \biggr) \ge
G_{{\xi}/{2}}(t,x)\sqrt{\frac{\xi t}{\nu
t-(\nu-\xi)\varepsilon}}
\]
and
\[
\int_\varepsilon^t \frac{e^{-{\mathit{l}_\rho^4 s}/{(4\nu
)}}}{\sqrt{s}}\,\ud s \ge
\frac{1}{\sqrt{t}} \int_\varepsilon^t e^{-{\mathit{l}_\rho
^4 s}/{(4\nu)}}
\,\ud s = \frac{4\nu}{\mathit{l}_\rho^4 \sqrt{t}} \bigl(e^{-{\mathit
{l}_\rho^4
\varepsilon}/{(4\nu)}}- e^{-{\mathit{l}_\rho^4 t}/{(4\nu)}} \bigr).
\]
Since $t\ge\varepsilon$,
\[
\bigl( I_{0,l}^2\star K \bigr) (t,x) \ge
\frac{2 a^2 \sqrt{\nu}}{\mathit{l}_\rho^4 \sqrt{\pi t}} G_{
{\xi
}/{2}}(t,x) \sqrt{\frac{\xi t}{\nu t-(\nu-\xi)\varepsilon}}
\bigl(e^{{\mathit{l}_\rho^4 (t-\varepsilon)}/{(4\nu)}}-1 \bigr).
\]
Thus,
\begin{eqnarray*}
&&\limsup_{t\rightarrow+\infty} \frac{1}{t}\sup_{|x|>\alpha t}
\log f(t,x)\\
&&\qquad \ge \liminf_{t\rightarrow+\infty} \frac{1}{t}\sup
_{|x|>\alpha t} \log f(t,x)
\\
&&\qquad\ge \lim_{t\rightarrow+\infty} \frac{1}{t}\sup_{|x|>\alpha t}
\log \bigl( e^{{\mathit{l}_\rho^4
(t-\varepsilon)}/{(4\nu)}} G_{{\xi}/{2}}(t,x) \bigr)= \frac{\mathit{l}_\rho^4}{4\nu} -
\frac{\alpha^2}{\xi}.
\end{eqnarray*}
The right-hand side is positive for $\alpha\le\sqrt{\xi/\nu} \mathit{l}_\rho^2/2$.
Since $\xi\in\,]0,\nu[ $ is arbitrary, we conclude that
$\underline{\lambda}(2) \ge\mathit{l}_\rho^2/2$.

As for the case $\underline{\varsigma}\ne0$, for all $\mu\in
\calM_{H,+} (\R )$, $f(t,x)\ge\underline{\varsigma}^2
\calH(t)$, and hence
\[
\liminf_{t\rightarrow\infty}\frac{1}{t}\sup_{|x|\ge\alpha t}
\log f(t,x) \ge \lim_{t\rightarrow\infty}\frac{1}{t} \log \bigl(
\underline {\varsigma}^2 \calH(t) \bigr) = \frac{\mathit{l}_\rho^4}{4\nu}>0\qquad
\mbox{for all $\alpha>0$}.
\]
Therefore, $\underline{\lambda}(2)=\infty$, which implies
$\overline{\lambda}(2)=\infty$. This proves part (1).
\end{pf*}

%s4.2 #&#
\subsection{Proof of the upper bound}\label{SS2:ExpInd-Up}
We need two lemmas.

%le4.3 #&#
\begin{lemma}\label{L2:MaxGGE}
For all $t>0$, $s>0$, $\sd>0$ and $x\in\R$, denote
\begin{eqnarray*}
&&H(x;\sd,t,s)\\
&&\qquad:=\sup_{(z_1,z_2)\in\R^2} G_{2\nu}(s,z_2-z_1)
G_{{\nu}/{2}} \biggl(t,x-\frac{z_1+z_2}{2} \biggr) \exp \bigl(-
\sd|z_1|-\sd|z_2| \bigr).
\end{eqnarray*}
Then
\[
H(x;\sd,t,s) \le \cases{\displaystyle \frac{1}{2\pi\nu\sqrt{t s}} \exp \biggl(-\frac{x^2}{\nu t}
\biggr), &\quad $\mbox{if $|x|\le\nu\sd t$,}$
\cr
\displaystyle\frac{1}{2\pi\nu\sqrt{t s}} \exp \bigl(-2
\sd|x| + \nu\sd^2 t \bigr), & \quad $\mbox{if $|x|\ge\nu\sd t $.}$}
\]
In particular, for all $x\in\R$, $\sd>0$, $t>0$ and $s>0$,
%
%e4.1 #&#
\begin{equation}
\label{E2:Hxbts} H(x;\sd,t,s) \le
\frac{1}{2\pi\nu\sqrt{t s}} \exp \bigl(-2\sd|x|+\nu
\sd^2 t \bigr).
\end{equation}
\end{lemma}
\begin{pf}
We only need to maximize over $(z_1,z_2)\in\R^2$ the exponent
\[
- \frac{(z_1-z_2)^2}{4\nu s} -\frac{ (x-{(z_1+z_2)}/{2} )^2}{\nu t} -\sd|z_1| -
\sd|z_2|.
\]
By the change of variables $u=\frac{z_1-z_2}{2}$, $w=\frac
{z_1+z_2}{2}$, we
have that
\[
\frac{u^2}{\nu s} + \frac{(x-w)^2}{\nu t} + \sd \bigl(|u+w|+|u-w| \bigr) \ge
\frac{(x-w)^2}{\nu t} +2 \sd|w| := f(w).
\]
Hence, we only need to minimize $f(w)$ for $w\in\R$.
Hence,
\[
\min_{w\in\R} f(w) = \cases{\displaystyle \frac{x^2}{\nu t}, & \quad $\mbox{if
$|x|\le\nu\sd t$}$,
\cr
2\sd|x| - \nu t \sd^2, &\quad  $\mbox{if $|x|\ge\nu
\sd t$}$.}
\]
This also implies \eqref{E2:Hxbts} since $\frac{x^2}{\nu t} \ge2 \sd
|x| - \nu
t \sd^2$ for all $x\in\R$.
\end{pf}

%le4.4 #&#
\begin{lemma}\label{L2:InDt-G}
Suppose $\mu\in\calM_G^{\sd} (\R )$ with
$\sd>0$. Set
$C=\int_\R e^{\sd|x|} |\mu|(\ud x)$.
Let $K(t,x)=G_{\nu/2}(t,x) h(t)$
for some nonnegative function $h(t)$. Then
%
%e4.2 #&#
%e4.3 #&#
\begin{eqnarray}
\label{E2:J0Bd-G} J_0^2(t,x) & \le&\frac{C^2}{2\pi\nu t}
e^{-2 \sd|x|+ \nu\sd^2 t},
\\
\bigl(J_0^2 \star K \bigr) (t,x) &\le&
\frac{C^2}{2\pi\nu\sqrt{t}} e^{-2 \sd|x|+ \nu\sd^2 t} \int_0^t
\frac{h(t-s)}{\sqrt{s}}\,\ud s. \label{E2:J0KBd-G}
\end{eqnarray}
\end{lemma}
\begin{pf}
Clearly,
\[
\bigl\llvert J_0(t,x)\bigr\rrvert \le \Bigl(\sup
_{y\in\R} G_\nu (t,x-y ) e^{-\sd|y|} \Bigr) \int
_\R e^{\sd
|x|}|\mu|(\ud y).
\]
The supremum is determined by minimizing
$\frac{(x-y)^2}{2\nu t} +\sd|y|$ over $y\in\R$, which has been done
in the
proof of Lemma~\ref{L2:MaxGGE}, and \eqref{E2:J0Bd-G} follows.
The proof of \eqref{E2:J0KBd-G} is similar to Lemma~\ref{L2:InDt}. By
\eqref{E2:InDt-Mid} and Lemma~\ref{L2:MaxGGE},
\begin{eqnarray*}
\bigl(J_0^2 \star K \bigr) (t,x) &\le&\int
_0^t H(x;\sd,t,s) h(t-s) \,\ud s
\iint_{\R^2}e^{\sd
|z_1|+\sd|z_2|}|\mu|(\ud z_1) |\mu|(\ud
z_2)
\\
& =& \biggl(\int_\R e^{\sd|x|}|\mu|(\ud x)
\biggr)^2\int_0^t H(x;\sd,t,s)
h(t-s) \,\ud s.
\end{eqnarray*}
Then apply \eqref{E2:Hxbts}.
\end{pf}

Note that one can apply the bound in
\eqref{E2:InDt} to \eqref{E2:SecMom-Up} and then Lemma~\ref{L2:InDt-G}
to get
$\overline{\lambda}(2)\le
\mathit{L}_\rho^2/\sqrt{2}$.
But we need a better estimate with $\sqrt{2}$ replaced by $2$. This gap
is due
to the factor $2$ in $J_0^*(2t,x)$ of \eqref{E2:InDt}, coming from Lemma~\ref{L2:Split}, which is not optimal.

\begin{pf*}{Proof of Theorem~\ref{T2:Growth}(2)}
Assume that $\Vip=0$. We first consider $\overline{\lambda}(2)$. Set
$f(t,x)=\E(u(t,x)^2)$. Fix $\sd>0$.
Without loss of generality, assume that $\mu\in\calM_G^{\sd}(\R)$ is
nonnegative; otherwise, simply replace all $\mu$ below by $|\mu|$.
By \eqref{E:K},
\[
\overline{\calK}(t,x) \le h(t) G_{{\nu}/{2}}(t,x)\qquad \mbox{with }
h(t)=
\frac{\mathit{L}_\rho^2}{\sqrt{4\pi\nu t}} + \frac{\mathit{L}_\rho^4}{2\nu}
\exp \biggl(\frac{\mathit
{L}_\rho^4 t}{4\nu}
\biggr),
\]
so \eqref{E2:SecMom-Up} implies that
\[
f(t,x)\le J_0^2(t,x)+ \bigl(J_0^2(
\cdot,\circ)\star G_{{\nu}/{2}}(\cdot,\circ) h(\cdot) \bigr) (t,x).
\]
By Lemma~\ref{L2:InDt-G}, \eqref{E2:BetaInt} and \eqref{E2:ExpDivSqrt},
\begin{eqnarray*}
f(t,x)\le \frac{C^2}{2\pi\nu t} e^{- 2 \sd|x|+\nu\sd^2 t}+\frac{C^2\mathit{L}_\rho^2}{2
\pi^{1/2}\nu^{3/2}\sqrt{t}} \biggl(
\frac{1}{2}+ e^{{\mathit{L}_\rho^4 t}/{(4\nu)} } \biggr) e^{-2\sd|x|+\nu
\sd^2 t}.
\end{eqnarray*}
Therefore, for $\alpha>0$,
\begin{eqnarray*}
\sup_{|x|> \alpha t} f(t,x)\le  \frac{C^2}{2\pi\nu t} e^{\sd^2 \nu t- 2 \sd\alpha t}+
\frac{C^2\mathit{L}_\rho^2}{2
\pi^{1/2}\nu^{3/2}\sqrt{t}} \biggl(\frac{1}{2}
+ e^{{\mathit{L}_\rho^4 t}/{(4\nu)}} \biggr)
e^{-2\sd\alpha t+\nu\sd^2 t}.
\end{eqnarray*}
Now, the exponential growth rate comes from the second term, and
\begin{eqnarray*}
\frac{\mathit{L}_\rho^4 t }{4\nu} - 2 \sd\alpha t+ \nu\sd^2 t<0
\quad\Longleftrightarrow \quad\alpha>\frac{\sd\nu}{2} + \frac{\mathit{L}_\rho^4}{8\nu\sd}.
\end{eqnarray*}
Therefore,
\[
\overline{\lambda}(2)\le\inf \biggl\{\alpha>0\dvtx\limsup_{t\rightarrow\infty}
\frac{1}{t}\sup_{|x|>\alpha t}\log f(t,x)<0 \biggr\}\le
\frac{\sd\nu}{2} + \frac{\mathit{L}_\rho^4}{8\nu\sd
}.
\]
Notice that
the function $\sd\mapsto\frac{\sd\nu}{2} + \frac{\mathit{L}_\rho
^4}{8\nu
\sd}$ is decreasing for $\sd\le\frac{\mathit{L}_\rho^2}{2\nu}$ and
increasing for
$\sd\ge\frac{\mathit{L}_\rho^2}{2\nu}$, with minimum value
$\mathit{L}_\rho^2/\nu$,
and $\calM_G^{\sd} (\R ) \subseteq
\calM_G^{\mathit{L}_\rho^2/(2\nu)} (\R )$ for $\sd\ge
\frac{\mathit{L}_\rho^2}{2\nu}$. This yields the desired upper bound.

%f1 #&#
\begin{figure}[b]

\includegraphics{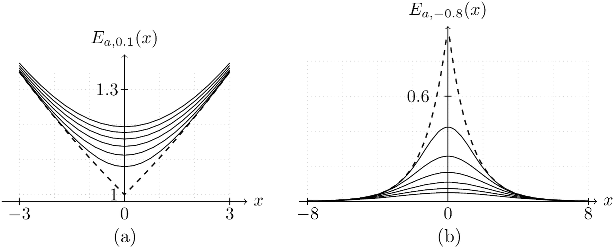}

\caption{The dashed lines in both figures denote the graph of $e^{\sd|x|}$.
The
solid lines
from bottom to top are $E_{a,\sd}(x)$ with the parameter $a$ ranging
from $1$ to
$6$ for Figure \protect\ref{F2:E}\textup{(a)} and from $6$ to $1$ for
Figure \protect\ref{F2:E}\textup{(b)},
which are representative of the cases $\beta > 0$ and $\beta <
0$, respectively. The parameter $\sd$ controls the asymptotic behavior near infinity while
both $a$ and $\sd$ determine how the function $e^{\sd|x|}$ is smoothed
at zero. The smaller $a$ is, the closer $E_{a,\sd}(0)$ is to $1$.}
\label{F2:E}
\end{figure}

Now fix an even integer $p\ge2$. Because the definition of
$\overline{\lambda}(p)$ differs from that of $\overline{\lambda
}(2)$ by
the use
of $\Vert u(t,x)\Vert_p^2$, we only need to make the following changes
in the
above proof:
(1) Replace $f(t,x)$ by $\Vert u(t,x)\Vert_{p}^2$.
(2) As in \eqref{E2:SecMom-Up}, replace $J_0^2(t,x)$ by $2J_0^2(t,x)$.
(3) Replace $\overline{\calK}(t,x)$ by
$\widetilde{\calK}_{p}(t,x)$, which is
equivalent to replacing $\mathit{L}_\rho$ everywhere by $\sqrt{2}
z_{p}\mathit{L}_\rho$.
This proves (2).
\end{pf*}

%s4.3 #&#
\subsection{Proof of Proposition \texorpdfstring{\protect\ref{P2:Ex-Exp}}{2.14}}
\label{SS2:ExpInd-Spe}

For $a>0$ and $\sd\in\R$, define
%
%e4.4 #&#
\begin{equation}
\label{E2:Eac} E_{a,\sd}(x):= e^{-\sd x} \Phi \biggl(
\frac{a \sd-x}{\sqrt{a}} \biggr) + e^{\sd x} \Phi \biggl(\frac{a \sd+x}{\sqrt{a}}
\biggr),
\end{equation}
which is a smooth version of the continuous function $e^{\sd|x|}$ (see
Figure~\ref{F2:E}).
Equivalently, by Proposition~\ref{P2:E}(ii),
%
%e4.5 #&#
\begin{equation}
\label{E2:Eac2} E_{a,\sd}(x) = e^{-\sd^2 a/2} \bigl(e^{\sd|\cdot|}*
G_a(1,\cdot ) \bigr) (x).
\end{equation}
Note that the function $ (e^{\sd|\cdot|}*G_\nu(t,\cdot) )(x)$
is the
solution to the homogeneous heat equation \eqref{E2:Heat-home} with initial
condition $\mu(\ud x)=e^{\sd|x|}\,\ud x$.
See Proposition~\ref{P2:E} below for its properties.

%\begin{figure}[ht]
%\centering
%\subfloat[The case $\sd>0$ ($\sd=0.1$).]{
%\label{F2:subfig1}
%\includegraphics[scale=1]{Function-E-up}
%}
%
%\subfloat[The case $\sd<0$ ($\sd=-0.8$)]{
%\label{F2:subfig2}
%\includegraphics[scale=1]{Function-E-down}
%}
% x
%\end{figure}
%
Recall (\cite{NIST2010}, Equation~7.12.1) that
%
%e4.6 #&#
\begin{eqnarray}
\label{E2:Phi-Asy}
1-\Phi(x) &\sim&\frac{e^{-x^2/2}}{\sqrt{2\pi} x}\quad \mbox{as }x\rightarrow+
\infty\quad \mbox{and}
\nonumber
\\[-8pt]
\\[-8pt]
\nonumber
\Phi(x) &\sim&\frac{e^{-x^2/2}}{\sqrt{2\pi} |x|} \qquad\mbox{as }
x\rightarrow-\infty.
\end{eqnarray}

\begin{pf*}{Proof of Proposition~\ref{P2:Ex-Exp}}
The fact that $\overline{\lambda}(2)$ is bounded above by the
expression in
\eqref{E2:Ex-Exp} follows from Theorem~\ref{T2:Growth} since
$\mu\in\calM_{G,+}^{\sd'}(\R)$, for any $\sd'<\sd$.
We now establish the corresponding lower bound on $\underline{\lambda}(2)$.
%For $\sd> \frac{\lambda^2}{2\nu}$, the lower bound follows from
%Theorem~\ref{T2:Growth} (3), so we assume that $\sd< \frac{
%\lambda^2}{2\nu}$.
Set $f(t,x)=\E(u(t,x)^2)$.
If $\mu(\ud x) = e^{-\sd|x|}\,\ud x$ with $\sd>0$, then by
\eqref{E2:Eac2}, $J_0(t,x) = e^{\sd^2\nu t /2}
E_{\nu t,-\sd}(x)$ and by Proposition~\ref{P2:E}(iv),
%
%e4.7 #&#
\begin{equation}
\label{E2_:J2-twoBds} J_0^2(t,x) \geq e^{ \sd^2\nu t}
\Phi^2 (-\sd\sqrt{\nu t} )E_{\nu t,-2 \sd}(x). % \le e^{\sd^2\nu t-2\sd|x|}.
\end{equation}
By \eqref{E2:Eac2} and the lower bound in \eqref{E2_:J2-twoBds},
\[
J_0^2(t,x) \ge e^{ -\sd^2\nu t} \Phi^2 (-
\sd\sqrt{\nu t} ) \bigl(e^{-2
\sd
|\cdot|}* G_\nu(t,\cdot) \bigr) (x).
\]
Thus, by \eqref{E2:SecMom} and the fact that
$\calK(t,x)\ge\frac{\lambda^4}{4\nu}G_{\nu/2}(t,x)
\exp (\frac{\lambda^4 t}{4\nu} )$,
\begin{eqnarray*}
f(t,x)&\ge& \int_0^t e^{-\sd^2 \nu(t-s)}
\Phi^2 \bigl(-\sd\sqrt{\nu (t-s)} \bigr)\frac{\lambda^4}{4\nu}
\\
&&{}\times e^{{\lambda^4 s}/{(4\nu)}} \biggl(e^{-2 \sd|\cdot|}*G_\nu \biggl(t-
\frac{s}{2},\cdot \biggr) \biggr) (x) \,\ud s.
\end{eqnarray*}
Noticing that by Proposition~\ref{P2:E}(ii) and (vi),
\begin{eqnarray*}
&&\biggl(e^{-2 \sd|\cdot|}*G_\nu \biggl(t-\frac{s}{2},\cdot
\biggr) \biggr) (x)\\
&&\qquad = e^{2 \sd^2\nu(t-{s}/{2})}
 E_{\nu(t-{s}/{2}),-2 \sd}(x) \ge e^{2 \sd^2\nu(t-s/2)}
E_{{\nu t}/{2},-2 \sd}(x),
\end{eqnarray*}
we have that
\begin{eqnarray*}
f(t,x) \ge E_{{\nu t}/{2},-2 \sd}(x) e^{ \sd^2\nu t} \int_0^t
\frac
{\lambda
^4}{4\nu} \Phi^2 \bigl(-\sd\sqrt{\nu(t-s)} \bigr)
e^{{\lambda^4
s}/{(4\nu)}} \,\ud s.
\end{eqnarray*}
Choose an arbitrary constant $c\in[0,1[$. The integral above is
bounded by
\begin{eqnarray*}
\int_0^t \frac{\lambda^4}{4\nu}
\Phi^2 \bigl(-\sd\sqrt{\nu(t-s)} \bigr) e^{{\lambda^4
s}/{(4\nu)}} \,\ud s &\ge&
\Phi^2 \bigl(-\sd\sqrt{\nu(1-c)t} \bigr) \int_{c t}^t
\frac{\lambda^4}{4\nu} e^{{\lambda^4 s}/{(4\nu)}} \,\ud s
\\
&=& \Phi^2 \bigl(-\sd\sqrt{\nu(1-c)t} \bigr) \bigl(e^{{\lambda^4
t}/{(4\nu)}}-
e^{{c \lambda^4
t}/{(4\nu)}} \bigr).
\end{eqnarray*}
Hence,
\[
f(t,x)\ge E_{{\nu t}/{2},-2 \sd}(x) e^{ \sd^2\nu t} \Phi^2 \bigl(-\sd
\sqrt {\nu (1-c)t} \bigr) \bigl( e^{{\lambda^4 t}/{(4\nu)}}
 -e^{{c\lambda^4 t}/{(4\nu)}} \bigr).
\]
By Proposition~\ref{P2:E}(v), for $\alpha>0$,
\[
\sup_{|x|> \alpha t} E_{{\nu t}/{2},-2 \sd}(x) = E_{{\nu t}/{2},-2
\sd}(\alpha
t).
\]
Notice that
\begin{eqnarray*}
&&E_{{\nu t}/{2},-2 \sd}(\alpha t) \\
&&\qquad= e^{2 \sd\alpha t} \Phi \biggl(- \biggl[2 \sd\sqrt{
\frac{\nu}{2}} + \alpha\sqrt{\frac
{2}{\nu
}} \biggr] \sqrt{t} \biggr)+
e^{-2 \sd\alpha t} \Phi \biggl( \biggl[\alpha\sqrt{\frac{2}{\nu}}-2 \sd
\sqrt{\frac
{\nu
}{2}} \biggr] \sqrt{t} \biggr).
\end{eqnarray*}
If $\alpha\sqrt{\frac{2}{\nu}}-2 \sd\sqrt{\frac{\nu}{2}} \ge
0$, that is,
$\alpha\ge\sd\nu$,
then by \eqref{E2:Phi-Asy}, the second term dominates and so for large $t$,
\[
E_{{\nu t}/{2},-2 \sd}(\alpha t) \ge\tfrac{1}{4}e^{-2 \sd\alpha t}.
\]
Otherwise, if $\alpha<\sd\nu$, then by \eqref{E2:Phi-Asy}, for
large $t$,
\[
e^{\pm2 \sd\alpha t} \Phi \biggl(\mp \biggl[\frac{\alpha}{\sqrt{\nu/2}}\pm2 \sd\sqrt {\nu
/2} \biggr] \sqrt{t} \biggr) \approx
\frac{\sqrt{\nu} \exp \{- (\sd^2\nu+ {\alpha
^2}/{\nu
} )
t \}}{2\sqrt{\pi}\llvert \alpha\pm\sd\nu\rrvert \sqrt{t}}.
\]
So $E_{\nu t/2,-2 \sd}(\alpha t)$ has a lower bound with the exponent
$-2\sd\alpha t$ if $\alpha\ge\sd\nu$, and $-(\sd^2 \nu+
\alpha^2/\nu) t$ if $\alpha<\sd\nu$.
For large $t$, by \eqref{E2:Phi-Asy}, the function $t\mapsto
\Phi^2 (-\sd\sqrt{\nu(1-c)t} )$ contributes to an exponent
$\sd^2 \nu
(c-1)t$.
Therefore,
\[
\lim_{t\rightarrow\infty} \frac{1}{t} \sup_{|x|> \alpha t}
\log f(t,x) \ge \cases{\displaystyle c \sd^2\nu+ \frac{\lambda^4}{4\nu} - 2
\sd\alpha,&\quad
$\mbox{if $\alpha \geq \sd\nu$,}$
\cr
\displaystyle (c-1)\sd^2\nu+
\frac{\lambda^4}{4\nu} -\frac{\alpha^2}{\nu}, & \quad$\mbox{if $\alpha< \sd\nu$.}$}
\]
%
% Suppose $\sd\ge\frac{\lambda^2}{2\nu}$. Since
% $(c-1)\sd^2\nu+\frac{\lambda^4}{4\nu}-\frac{\alpha^2}{\nu}>0
%\Leftrightarrow
% \alpha<\sqrt{\frac{\lambda^4}{4}-(c-1)\sd^2\nu^2}$, and
% $\frac{\lambda^2}{2\sqrt{2-c}}\le\sd\nu$, it follows that
% $\underline{\lambda}(2)\ge\sqrt{\frac{\lambda^4}{4}-(c-1)\sd^2
%\nu^2}$ in this
% case.

We now consider two cases. First, suppose that $\sd< \frac{\lambda
^2}{2\nu\sqrt{2-c}}$. This inequality is equivalent to $\frac{c \nu
\sd}{2}+ \frac{\lambda^4}{8\nu\sd}> \sd\nu$, and
\[
c\sd^2\nu+\frac{\lambda^4}{4\nu} -2 \sd\alpha>0 \quad\Leftrightarrow\quad \alpha <
\frac{c \nu\sd}{2}+ \frac{\lambda^4}{8\nu\sd}.
\]
Therefore, $\underline{\lambda}(2)\ge
\frac{c \nu\sd}{2}+ \frac{\lambda^4}{8\nu\sd}$ in this first case.
Second, suppose that $\sd\geq\frac{\lambda^2}{2\nu\sqrt{2-c}}$. This
inequality is equivalent to $\sqrt{\frac{\lambda^4}{4}+(c-1)\sd
^2\nu
^2} \leq\sd\nu$, and
\[
(c-1)\sd^2\nu+ \frac{\lambda^4}{4\nu} -\frac{\alpha^2}{\nu}>0
\quad\Leftrightarrow\quad \alpha< \sqrt{\frac{\lambda^4}{4}+(c-1)\sd^2
\nu^2}.
\]
Therefore, $\underline{\lambda}(2)\ge\sqrt{\frac{\lambda
^4}{4}+(c-1)\sd
^2\nu^2}$ in this second case.

Finally, since the constant $c$ can be
arbitrarily close to $1$,
this completes the proof.
%If $\alpha\ge\sd\nu$, then since $c\sd^2\nu+\frac{\lambda^4}{4\nu}
%-2 \sd\alpha>0 \Leftrightarrow\alpha<\frac{c \nu\sd}{2}+ \frac{
%\lambda^4}{8\nu\sd}$, and $\frac{c \nu\sd}{2}+ \frac{\lambda^4}{8
%\nu\sd}\ge\sd\nu\Leftrightarrow\sd\le\frac{\lambda^2}{2\nu
%\sqrt{2-c}}$, it follows that $\underline{\lambda}(2)\ge\frac{c \nu
%\sd}{2}+ \frac{\lambda^4}{8\nu\sd}$ if $\sd\le\frac{\lambda^2}{2\nu
%\sqrt{2-c}}$. If $\alpha<\sd\nu$, then since $(c-1)\sd^2\nu+ \frac{
%\lambda^4}{4\nu} -\frac{\alpha^2}{\nu}>0 \Leftrightarrow\alpha< \sqrt{
%\frac{\lambda^4}{4}+(c-1)\sd^2\nu^2}$, and $\sqrt{\frac{
%\lambda^4}{4}+(c-1)\sd^2\nu^2} < \sd\nu\Leftrightarrow\sd>\frac{
%\lambda^2}{2\nu\sqrt{2-c}}$, it follows that $\underline{\lambda}(2)
%\ge\sqrt{\frac{\lambda^4}{4}+(c-1)\sd^2\nu^2}$ if $\sd>\frac{
%\lambda^2}{2\nu\sqrt{2-c}}$. Finally, since the constant $c$ can be
%arbitrarily close to $1$, this completes the proof.
\end{pf*}
\begin{appendix}\label{app}
%s5 #&#
\section*{Appendix}\label{sec4}
%
%le5.1 #&#
%\setcounter{lemma}{0}
\begin{lemmas}\label{LA:Phi-E}
$\pi\int_0^t e^{\pi b^2 u} \Phi (\sqrt{2\pi b^2 u} ) \,\ud
u
= \frac{e^{\pi b^2 t} \Phi (\sqrt{2\pi b^2 t} ) }{b^2}
-\frac{1}{2b^2}
-\frac{\sqrt{t}}{b}$, $b\ne0$.
\end{lemmas}
\begin{pf}\hspace*{-6pt}
By integration by parts, the left-hand side equals\break 
$\frac{e^{\pi b^2 u} \Phi (\sqrt{2\pi b^2 u} ) }{b^2}
|_{u=0}^{u=t} - \frac{1}{b^2} \int_0^t \frac{b}{2\sqrt{s}} \,\ud s$.
\end{pf}

%le5.2 #&#
\begin{lemmas}\label{L2:LogRel}
For $0< a < b$, we have that
%
%e5.1 #&#
\setcounter{equation}{0}
\begin{equation}
\label{E2:log} \frac{\log(b/a)}{b-a}\ge\frac{1}{b}.
\end{equation}
The function $f(s)=(a-s)(b-s)\log\frac{b-s}{a-s}$ is nonincreasing over
$s\in
[0,a[$ with $\inf_{s\in[0,a[}f(s) = \lim_{s\rightarrow a}f(s)
=(b-a)\log(b-a)$
and $\sup_{s\in[0,a[}f(s) = f(0) =a b \log(b/a)$.
\end{lemmas}

\begin{pf}
Note that \eqref{E2:log} is equivalent to
the following statements:
\begin{eqnarray*}
\frac{-\log s}{1-s}\ge1,\qquad s\in\,]0,1[ \quad \Longleftrightarrow \quad
s-\log s\ge1, \qquad s
\in\,]0,1[.
\end{eqnarray*}
Let $g(s)=s-\log s$ with $s\in\,]0,1[$. Then $g(s)$ is nonincreasing since
$g'(s)=(s-1)/s<0$ for $s\in\,]0,1[ $. So
$g(s)\ge\lim_{s\rightarrow1} g(s)=1$. This proves \eqref{E2:log}.
As for the function $f(s)$, we only need to show that
\[
f'(s)=(b-a)-(a+b-2s)\log\frac{b-s}{a-s}\le0\qquad\mbox{for all $s
\in[0,a[$}.
\]
Let $g(s)=\frac{b-a}{a+b-2s}-\log\frac{b-s}{a-s}$. Then the above
statement is
equivalent to the inequality $g(s)\le0$ for all $s\in[0,a[ $. By
\eqref{E2:log}, we know that
\[
g(0)=\frac{b-a}{a+b}-\log\frac{b}{a}\le (b-a) \biggl(
\frac{1}{a+b}-\frac{1}{b} \biggr)\le0.
\]
So it suffices to show that
\[
g'(s) = \frac{2(b-a)}{(a+b-2s)^2} +\frac{1}{b-s}-
\frac{1}{a-s}\le 0 \qquad\mbox{for all $s\in[0,a[$}.
\]
After simplifications, this statement is equivalent to
\[
s^2 -(a+b) s + \frac{a^2+b^2}{2}\ge0 \qquad \mbox{for all $s\in[0,a[$} ,
\]
which is clearly true since the discriminant is $-(a+b)^2<0$. This completes
the proof.
\end{pf}

%pr5.3 #&#
\begin{propositionn}\label{P2:G-Margin}
Fix $(t,x)\in\R_+^*\times\R$. Set
\[
B_{t,x}= \bigl\{ \bigl(t',x' \bigr)\in
\R_+^*\times\R\dvtx0< t'\le t+\tfrac{1}{2}, \mbox{ and } \bigl
\llvert x'-x\bigr\rrvert \le1 \bigr\}.
\]
Then there exists $a=a_{t,x}>0$ such
that for all $ (t',x' )\in B_{t,x}$, $s\in[0,t']$ and
\mbox{$|y|\ge a$},
\[
G_\nu\bigl(t'-s,x'-y\bigr) \le
G_\nu(t+1-s,x-y).
\]
\end{propositionn}

\begin{pf}
Since $t+1-s$ is strictly larger than $t'-s$, the function $y\mapsto
G_\nu(t+1-s,x-y)$ has heavier tails than $y\mapsto
G_\nu(t'-s,x'-y)$. Solve the inequality $ G_\nu(t+1-s,x-y) \ge G_\nu
(t'-s,x'-y)$
with $t,t',x,x'$ and $s$ fixed, which is a quadratic inequality for $y$:
\[
-\frac{(x'-y)^2}{t'-s} +\frac{(x-y)^2}{t+1-s} \le\nu\log \biggl(\frac{t'-s}{t+1-s}
\biggr).
\]
Let $y_{\pm}(t,x,t',x',s)$ be the two solutions of the corresponding quadratic
equation, which are
\begin{eqnarray*}
&&\frac{1}{t +
1 - t'}\biggl((t+1-s)x'-x\bigl(t'-s\bigr)\\
&&\qquad{} \pm
\biggl [(t+1-s)\bigl(t'-s\bigr)\\
&&\qquad{}\times \biggl\{\bigl(x-x'\bigr)^2+\bigl(t+1-t'\bigr)\nu
\log\biggl (\frac{t+1-s}{t'-s} \biggr) \biggr\} \biggr]^{1/2}\biggr).
\end{eqnarray*}
Then a sufficient condition for the above inequality is $|y|\ge
|y_+|\vee
|y_-|$.
So we only need to show that
\[
\sup_{ (t',x' )\in B_{t,x}} \sup_{s\in[0,t']} \bigl|y_{+}
\bigl(t,x,t',x',s\bigr)\bigr|\vee \bigl|y_{-}
\bigl(t,x,t',x',s\bigr)\bigr|<+\infty.
\]
By Lemma~\ref{L2:LogRel},
the supremum over $s\in[0,t']$ of the quantity under the square root is
\[
t'(t+1) \biggl[\bigl(x-x'\bigr)^2+
\bigl(t+1-t'\bigr)\nu\log\frac{t+1}{t'} \biggr],
\]
so, using the fact that $\llvert x'-x\rrvert \le1$, we see that
\begin{eqnarray*}
\hspace*{-6pt}&&|y_{+}|\vee|y_{-}| \\
\hspace*{-6pt}&&\qquad\le\frac{(t+1)(|x|+1)+|x| t'+
 [t'(t+1) \{1+(t+1-t')\nu
\log ({(t+1)}/{t'} ) \} ]^{1/2}}{t+1-t'}.
\end{eqnarray*}
Finally, because $t' \in[0,t+1/2]$, this right-hand side is
bounded above by
\begin{eqnarray*}
&&2(t+1) \bigl(|x|+1\bigr)+|x| (2t+1)\\
&&\quad{}+ 2 \biggl[(t+1) \biggl((t+1/2)+t'(t+1)
\nu \log \biggl(\frac{t+1}{t'} \biggr) \biggr) \biggr]^{1/2}
\\
&&\qquad < (4t+3) \bigl(|x|+1\bigr)+2(t+1) \sqrt{1+\nu/e}=: a,
\end{eqnarray*}
since $\sup_{s\ge0} s\log\frac{t}{s} =
s\log\frac{t}{s} |_{s=t/e} =
\frac{t}{e}$ for all $t>0$. This completes the proof.
\end{pf}

%le5.4 #&#
\begin{lemmas}\label{L2:GG}
For all $t$, $s>0$ and $x$, $y\in\R$, we have that $ G_\nu^2(t,x) =
\frac{1}{\sqrt{4\pi\nu t}} G_{{\nu}/{2}}(t,x)$ and
$G_\nu(t,x)G_\nu (s,y )
= G_\nu (\frac{ts}{t+s},\frac{s x+t
y}{t+s} )
G_\nu (t+s,x-y )$.
\end{lemmas}

The proof of this lemma is straightforward and is left to the reader.

%le5.5 #&#
\begin{lemmas}\label{L2:Split}
For all $x$, $z_1$, $z_2\in\R$ and $t,s>0$, denote
$\bar{z} = \frac{z_1+z_2}{2}$, $\Delta z = z_1-z_2$. Then
$G_1 (t,x-\bar{z} )
G_1 (s,\Delta z )
\le\frac{(4t) \vee s}{\sqrt{t s}}
G_1  ((4t)\vee s,x-z_1 )
G_1  ((4t)\vee s,x-z_2 )$, where $a\vee
b:=\max(a,b)$.
\end{lemmas}

\begin{pf}
Since $(z_2-z_1)^2 +  [(x-z_1)+(x-z_2) ]^2 \ge
(x-z_1)^2+(x-z_2)^2$,
\[
G_1 (t,x-\bar{z} ) G_1 (s,\Delta z )\le
\frac{1}{2\pi\sqrt{t s}}
 e^{-{ ([(x-z_1)+(x-z_2) ]^2+(z_1-z_2)^2)}/{(2 ((4t)\vee
s))}}.
\]
\upqed\end{pf}

%le5.6 #&#
\begin{lemmas}\label{L2:IntHG}
$\int_0^t  (\calH(r)+1 )
G_{2\nu}(t-r,x)\,\ud r = \frac{1}{\lambda^2}
 (
e^{{(\lambda^4 t-2 \lambda^2
|x|)}/{(4 \nu)}} \times\break  \Erfc (\frac{|x|-\lambda^2 t}{2
\sqrt{\nu t}} )- \Erfc (\frac{|x|}{2 \sqrt{\nu t}} )
 )$, $t\ge0$.
\end{lemmas}
\begin{pf}
Let $\mu=\frac{\lambda^4}{4\nu}$.
By \cite{Erdelyi1954-I}, (27) on page~146] and
\cite{Erdelyi1954-I}, (5) on page~176, the Laplace
transform of the convolution
equals
\begin{eqnarray*}
&&\calL \bigl[G_{2\nu}(\cdot,x) \bigr](z) \calL \bigl[\calH(\cdot)+1
\bigr](z)\\
&&\qquad = \frac{1}{\sqrt{4\nu}} \frac{1}{\sqrt{z}}e^{-{|x|\sqrt
{z}}/{\sqrt{\nu}}} \biggl(
\frac{1}{z-\mu}+\frac{\sqrt{\mu}}{\sqrt{z}(z-\mu)} \biggr)
\frac{\exp (-({|x|}/{\sqrt{\nu}}) \sqrt{z} )}{
\sqrt{4 \nu
} z  (\sqrt{z}-\mu )}.
\end{eqnarray*}
Then apply the inverse Laplace transform
(see \cite{Erdelyi1954-I}, (14) on
page~246).
\end{pf}

% The next lemma was used in Remark~\ref{R2:TP-Lebesgue}.
%
%le5.7 #&#
\begin{lemmas}\label{L2:IntHG-BC}
$\int_0^t \,\ud r \frac{|x| e^{-{x^2}/{(4\nu r)}+{(t-r)}/{(4\nu)
}}}{\sqrt{\pi
\nu r^3}}\Phi (\sqrt{\frac{t-r}{2\nu}} )
= \exp (\frac{t-2 |x|}{4 \nu} ) \Erfc (\frac
{|x|-t}{\sqrt
{4\nu
t}} )$,
for all $t\ge0$ and $x\ne0$.
\end{lemmas}
\begin{pf}
Suppose that $x\ne0$.
Denote the integral by $I(t)$. Let
\[
f(t)=\frac{|x|}{\sqrt{\pi\nu
t^3}}e^{-{x^2}/{(4\nu t)}}\quad \mbox{and}\quad g(t)=e^{{t}/{(4\nu)}}\Phi
\bigl(\sqrt{(2\nu)^{-1}t} \bigr).
\]
Clearly, $I(t)$ is the convolution of $f$ and $g$. By
\cite{Erdelyi1954-I}, (28) on page~146,
\[
\calL[f](z) = 2\exp \bigl(-|x|\sqrt{z/\nu} \bigr).
\]
Notice $g(t) = (H(t) + 1)/2$ with $H(t)=\calH(t;\nu, 1)$.
By the calculations in Lemma~\ref{L2:IntHG},
\[
\calL[g](z) = \frac{1}{2(z-1/(4\nu))} + \frac{1}{4\sqrt{\nu
z} (z-1/(4\nu) )}.
\]
Hence,
\[
\calL[I](z) = \calL[f](z)\calL[g](z) =
\frac{e^{-|x| \sqrt{z/\nu}}}{\sqrt{z}  (\sqrt{z}-{1}/{(2\sqrt{\nu})} )}.
\]
Then apply the inverse Laplace transform (see
\cite{Erdelyi1954-I}, (16) on page~247).
\end{pf}

% The next lemma was used in Remark~\ref{R2:TP-Delta}.

%le5.8 #&#
\begin{lemmas}\label{L2:TP-Delta-BC}
\eqref{E2:TP-Delta-BC} equals $G_\nu(t,x) G_\nu (t,y )
+
\frac{1}{4\nu} G_{{\nu}/{2}} (t,\frac{x+y}{2} )\times\break 
\exp (\frac{t-2|x-y|}{4\nu} )
\Erfc (\frac{|x-y|- t }{\sqrt{4\nu t }} )$.
\end{lemmas}

\begin{pf}
After some simplifications, the integral in \eqref{E2:TP-Delta-BC} is
equal to
the following integral:
\begin{eqnarray*}
&&\frac{1}{4\pi\nu t} G_{{\nu}/{2}} \biggl(t,\frac{x+y}{2} \biggr) \int
_0^1 \,\ud s\frac{|x-y|}{\sqrt{s^3}} \exp \biggl(-
\frac
{(x-y)^2}{4\nu t
s} \biggr)
\\
&&\qquad{}\times \biggl( \frac{1}{\sqrt{1-s}} + \sqrt{\pi t/\nu} \exp \biggl(
\frac{t
(1-s)}{4\nu
} \biggr) \Phi \biggl(\sqrt{\frac{t(1-s)}{2\nu}} \biggr) \biggr).
\end{eqnarray*}
Denote this integral by $I_1(1) + I_2(1)$. Suppose that $x\ne y$ and let
\begin{eqnarray*}
f(s)&=& \frac{|x-y|}{s^{3/2}} \exp \biggl(-\frac{(x-y)^2}{4\nu t
s} \biggr),\qquad g(s) =
\frac{1}{\sqrt{s}},\\
  h(s)&=& \frac{\sqrt{\pi t}}{\sqrt{\nu}} \exp \biggl(\frac{ts}{4\nu}
\biggr)\Phi \biggl(\sqrt{\frac{ts}{2\nu
}} \biggr).
\end{eqnarray*}
Then by \cite{Erdelyi1954-I}, (28) on page~146,
and \cite{Erdelyi1954-I}, page~135,
\[
\calL[I_1](z) = \calL[f](z)\calL[g](z) = 2\pi\sqrt{\nu t}
\frac{\exp (-{|x-y|\sqrt{z}}/{\sqrt {\nu t}} )}{\sqrt{z}}.
\]
Apply the inverse Laplace transform (see \cite{Erdelyi1954-I}, (6) on
page~246),
\[
I_1(s) = \frac{2\sqrt{\pi\nu t}}{\sqrt{s}} \exp \biggl(-\frac{(x-y)^2}{4\nu s
t}
\biggr)\qquad \mbox{for $s>0$.}
\]
As for $I_2(s)$, by the calculation in Lemma~\ref{L2:IntHG-BC},
\[
\calL[h](z) = \frac{\sqrt{\pi
t}}{2\sqrt{\nu}} \biggl( \frac{1}{z-t/(4\nu)} +
\frac{\sqrt{t}}{2\sqrt{\nu z}
(z-t/(4\nu
) )} \biggr).
\]
Hence,
\[
\calL[I_2](z) =\calL[f](z) \calL[h](z) = \pi t
e^{-{|x-y|\sqrt{z}}/{\sqrt{\nu t}}}
\frac{1}{\sqrt{z} (\sqrt{z}-\sqrt{t/(4\nu)} )}.
\]
Then apply the inverse Laplace transform (see
\cite{Erdelyi1954-I}, (16) on page~247).
Finally, let $s=1$ and use Lemma~\ref{L2:GG}.
\end{pf}

%le5.9 #&#
\begin{lemmas}\label{L:IntIntGG}
For $\nu>0$, $\tau\ge t \ge0$ and $x,y\in\R$,
\begin{eqnarray*}
\int_t^\tau G_\nu(r,x)\,\ud r =
\frac{2|x|}{\nu} \biggl(\Phi \biggl(\frac{|x|}{\sqrt{\nu
\tau}} \biggr)-\Phi \biggl(
\frac{|x|}{\sqrt{\nu t}} \biggr) \biggr) + 2 \tau G_{\nu}(\tau,x)-2 t
G_{\nu}(t,x)
\end{eqnarray*}
and
\begin{eqnarray*}
&&\int_0^t \,\ud r\int_\R
\,\ud z  G_\nu(t-r,x-z)G_\nu (\tau-r,y-z )
\\
&&\qquad= \frac{|x-y|}{\nu} \biggl(\Phi \biggl(\frac{|x-y|}{\sqrt{\nu
(\tau+t)}} \biggr)-\Phi
\biggl(\frac{|x-y|}{\sqrt{\nu(\tau
-t)}} \biggr) \biggr)
\\
&&\qquad\quad{} + (\tau+t) G_{\nu} (\tau+t,x-y ) - (\tau-t) G_{\nu} (
\tau-t,x-y ).
\end{eqnarray*}
\end{lemmas}
\begin{pf}
Consider the first integral.
The case where $x=0$ is straightforward, so we assume that $x\ne0$. This
right-hand side is obtained by a change variable and integration by parts:
\begin{eqnarray*}
\int^\tau_t G_\nu(r,x)\,\ud r &=&
\frac{2|x|}{\nu} \int^{|x|/\sqrt
{\nu
t}}_{|x|/\sqrt{\nu\tau}}
\frac{1}{\sqrt{2\pi} u^2} e^{-u^2/2}\,\ud u\\
& =&\frac{2|x|}{\nu} \biggl(
\frac{e^{-u^2/2}}{\sqrt{2\pi}u} \bigg|^{|x|/\sqrt{
\nu\tau}}_{|x|/\sqrt{ \nu t}} -\int^{|x|/\sqrt{\nu t}}_{|x|/\sqrt{\nu\tau}}
\frac
{e^{-u^2/2}}{\sqrt
{2\pi}}\,\ud u \biggr).
\end{eqnarray*}
For the second integral, use the semigroup property to integrate over
$z$, and then apply the first integral.
\end{pf}

%le5.10 #&#
\begin{lemmas}\label{L2:IntGG}
For $t \geq0$ and $x, y\in\R$, we have that
\[
\int_0^t G_\nu(r,x)G_\sigma(t-r,y)
\,\ud r = \frac{1}{2\sqrt{\nu\sigma}} \Erfc \biggl(\frac{1}{\sqrt{2t}} \biggl(
\frac{|x|}{\sqrt{\nu}} +\frac{|y|}{\sqrt{\sigma} } \biggr) \biggr),
\]
where $\nu$ and $\sigma$ are strictly positive.
In particular, by letting $x=0$, we have that
\[
\int_0^t \frac{G_\sigma(t-r,y)}{\sqrt{2\pi\nu r}} \,\ud r =
\frac{1}{2\sqrt{\nu\sigma}}\Erfc \biggl(\frac{|y|}{\sqrt{2\sigma
t}} \biggr)\le \frac{\sqrt{\pi t}}{\sqrt{2\nu}}
G_{\sigma} (t,y ).
\]
\end{lemmas}
\begin{pf}
By \cite{Erdelyi1954-I}, (27) on page~146, the Laplace transform of
the integrand is
\[
\calL \bigl[G_\nu(\cdot,x) \bigr](z)\cdot \calL
\bigl[G_\sigma(\cdot,y) \bigr](z) = \frac{\exp (-\sqrt{2z}
({|x|}/{\sqrt{\nu}}+
{|y|}/{\sqrt
{\sigma}}
 ) ) } { 2\sqrt{\nu\sigma z^2}},
\]
and the conclusion follows by applying the inverse Laplace transform (see
\cite{Erdelyi1954-I}, (3) on page~245).
As for the special case $x=0$, use formula \cite{NIST2010}, (Equation~7.7.1,
page~162) to
write
\[
\Erfc(x) = \frac{2}{\pi} e^{-x^2} \int_0^\infty
\frac{e^{-x^2 r^2}}{1+r^2}\,\ud r \le\frac{2}{\pi} e^{-x^2} \int
_0^\infty\frac{1}{1+r^2}\,\ud r =
e^{-x^2}.
\]
\upqed\end{pf}

%pr5.11 #&#
\begin{propositionn}[{[Properties of $E_{a,\sd}(x)$, defined in \eqref
{E2:Eac2}]}]
\label{P2:E}
For $a>0$ and $\sd\in\R$,
\begin{longlist}[(iii)]
\item[(i)]$E_{a,0}(x) = 1$;
\item[(ii)] for $\nu>0$,
$ (e^{\sd|\cdot|} * G_{\nu}(t,\cdot) ) (x)
= e^{{\sd^2\nu t}/{2}}E_{\nu t,\sd}(x)$;
% % %
%
\item[(iii)] first and second derivatives:
\begin{eqnarray*}
E_{a,\sd}'(x) &=& - \sd e^{-\sd x} \Phi \biggl(
\frac{a \sd-x}{\sqrt{a}} \biggr) + \sd e^{\sd x} \Phi \biggl(\frac{a \sd+x}{\sqrt{a}}
\biggr),
\\
E_{a,\sd}''(x)&=& \sd\sqrt{
\frac{2}{\pi a}} e^{-{(a^2 \sd
^2+x^2)}/{(2a)}} + \sd^2 E_{a,\sd}(x);
\end{eqnarray*}
%
% % %
%
\item[(iv)] for $\sd>0$, $e^{\sd|x|}\le E_{a,\sd}(x) < e^{\sd x}+e^{-\sd x}$;
for $\sd<0$, $\Phi (\sqrt{a} \sd )E_{a,2 \sd
}^{1/2}(x)\le
E_{a,\sd}(x) \le e^{-|\sd
x|}$;
% % %
%
\item[(v)] for $\sd>0$, $x\mapsto E_{a,\sd}(x)$ is strictly convex and
$\inf_{x\in\R}E_{a,\sd}(x) = \break E_{a,\sd}(0) =
2\Phi(\sd\sqrt{a})>1$,
with $E_{a,\sd}''(0)=\sd\sqrt{\frac{2}{\pi
a}}e^{-{\sd^2 a}/{2}}+2\sd^2\Phi(\sd\sqrt{a})>0$;
for $\sd<0$, the function $E_{a,\sd}(x)$ is decreasing for $x\ge0$ and
increasing
for $x\le0$, and it therefore achieves its global maximum at zero:
$\sup_{x\in\R}E_{a,\sd}(x) = E_{a,\sd}(0) = 2\Phi(\sd\sqrt{a})<1$,
with $E_{a,\sd}''(0)=\sd\sqrt{\frac{2}{\pi
a}}e^{-{\sd^2 a}/{2}}+2\sd^2\Phi(\sd\sqrt{a})\le0$;
% % % % % % % % % %
%
\item[(vi)] concerning $a\mapsto E_{a,\sd}(x)$,
\[
\frac{\partial E_{a,\sd}(x)}{\partial a}= \frac{\sd}{\sqrt{2\pi a}} \exp \biggl(-\frac{a^2 \sd^2+x^2}{2a}
\biggr).
\]
Hence, for all $x\in\R$, then the function $a\mapsto E_{a,\sd}(x)$ is
nondecreasing for $\sd>0$ and nonincreasing for $\sd<0$.
\end{longlist}
\end{propositionn}
\begin{pf}
(i) Is trivial. (ii) Follows from a direct calculation. (iii) Is
routine. We now
prove (iv).
Suppose that $\sd<0$.
We first prove the upper bound. Since $x\mapsto
E_{a,\sd}(x)$ is an even function, we shall
only consider $x\ge0$.
We need to show that for $x\ge0$
\[
e^{-\sd x}\Phi \biggl(\frac{a \sd-x}{\sqrt{a}} \biggr)+ e^{\sd x}\Phi
\biggl(\frac{a \sd+x}{\sqrt{a}} \biggr)\le e^{\sd x}
\]
or equivalently from the fact that $1-\Phi (\frac{a \sd
+x}{\sqrt
{a}} )
=
\Phi (\frac{-a \sd-x}{\sqrt{a}} )$,
\[
F(x):= e^{\sd x}\Phi \biggl(\frac{-a \sd-x}{\sqrt{a}} \biggr)- e^{-\sd x}
\Phi \biggl(\frac{a \sd-x}{\sqrt{a}} \biggr)\ge0.
\]
This is true since
\[
F'(x) =\sd e^{\sd x}\Phi \biggl(\frac{-a \sd-x}{\sqrt{a}}
\biggr)+ \sd e^{-\sd x}\Phi \biggl(\frac{a \sd-x}{\sqrt{a}} \biggr) \le0
\]
and $\lim_{x\rightarrow+\infty} F(x) = 0$ by applying l'H\^
{o}pital's rule.
Note that $F(0)=  \break \Phi(-\sqrt{a}\sd)-\Phi(\sqrt{a}\sd)>0$ since
$\sd<0$.

As for the lower bound, when $\sd<0$, we have that
\begin{eqnarray*}
E_{a,\sd}^2(x)&=& \biggl[ e^{-\sd x}\Phi \biggl(
\frac{a \sd-x}{\sqrt{a}} \biggr) + e^{\sd x}\Phi \biggl(\frac{a \sd+x}{\sqrt{a}}
\biggr) \biggr]^2
\\
&\ge& e^{-2|\sd x|} \Phi^2 \biggl(\frac{a \sd+|x|}{\sqrt{a}} \biggr) \ge
e^{-2|\sd x|} \Phi^2 (\sqrt{a}\sd ).
\end{eqnarray*}
Then the lower bound follows from the fact that $e^{-2|\sd x|}\ge
E_{a,2\sd}(x)$.
As for the first part of (iv) where $\sd>0$, the upper bound holds since
$\Phi(\cdot)<1$. The lower bound is a consequence of the upper
bound with
$\sd<0$ and the equality $E_{a,\sd}(x)=e^{\sd x}+e^{-\sd x}-E_{a,-\sd}(x)$,
which follows from \eqref{E2:Eac}.
Now consider (v). We first consider the case $\sd>0$. By (iii),
$E_{a,\sd}''(x)> 0$ for all $x\in\R$, hence $x\mapsto E_{a,\sd}(x)$
is strictly
convex.
By \eqref{E2:Eac2},
\[
\frac{\ud}{\ud x} E_{a,\sd}(x) = \sd e^{-{a\sd^2}/{2}}\int
_0^\infty e^{\sd
y} \bigl(
G_a(1,x-y)-G_a(1,x+y) \bigr)\,\ud y.
\]
Clearly, if $x\ge(\le) 0$, then $G_a(1,x-y)-G_a(1,x+y)\ge(\le)
0$ for all $y\ge0$. Hence, $\frac{\ud}{\ud x} E_{a,\sd}(x)\ge(\le)
0$ if
$x\ge(\le)
0$ and the global minimum is achieved at $x=0$. Similarly, for $\sd
<0$, we have
$\frac{\ud}{\ud x} E_{a,\sd}(x)\le(\ge) 0$ if $x\ge(\le)
0$ and the global maximum is taken at $x=0$, which then implies that
$E_{a,\sd}''(0)\le0$ [note that by (iii), $E_{a,\sd}''(x)$ exists].
As for (vi),
\[
\frac{\partial}{\partial a}e^{\mp\sd x}\Phi \biggl(\frac{a
\sd\mp x}{\sqrt{a}} \biggr) =
\frac{a \sd\pm x}{2 a^{3/2}\sqrt{2\pi}} \exp \biggl(-\frac{a^2\sd^2+x^2}{2a} \biggr).
\]
Adding these two terms proves the formula for $\frac{\partial
E_{a,\sd}(x)}{\partial a}$. The rest is
clear.
\end{pf}
\end{appendix}

\section*{Acknowledgements}
%\addcontentsline{toc}{section}{Acknowledgements}
The authors thank Daniel Conus, Davar Khoshnevisan, and Roger Tribe for
stimulating
discussions and Leif D\"oring for a discussion that led to Remark~\ref{R:Transform}. The authors also thank two anonymous referees for a careful
reading of this paper and many useful suggestions.

% \bibliographystyle{alpha}
%\addcontentsline{toc}{section}{Bibliography}
%
% imsref loaded by akundreckaite, 2014-09-05 08:49:32
% imsref loaded by akundreckaite, 2014-09-05 09:44:33
%

% \textit{URL:} \url{http://prob.epfl.ch}

%\begin{appendix}
%\section{}
%\end{appendix}

% zodis "Acknowledgments" paliekamas pagal autoriu
%\section*{Acknowledgments}

%\begin{supplement}[id=suppA]
%\sname{Supplement A}
%\stitle{}
%\slink[doi]{10.1214/00-AOPXXXXSUPP} %[doi,text={...}] - jei reikia
%suskaldyti doi
%\sdatatype{.pdf}
%\sfilename{aopXXXX\_supp.pdf}
%\sdescription{}
%\end{supplement}

%\begin{thebibliography}{99}
%\bibitem[\protect\citeauthoryear{}{}]{r1}
%\bibitem{r1}
%\end{thebibliography}

\printaddresses
\end{document}